\documentclass[11pt]{amsart}

\usepackage{mathptmx}
\usepackage{amssymb, amsmath}
\usepackage{rlepsf}
\usepackage{graphicx}
\usepackage{amscd}
\usepackage{hyperref}

\newtheorem{thm}{Theorem}[section]

\newtheorem{rem}[thm]{Remark}

\newtheorem{lma}[thm]{Lemma}

\newenvironment{pf}{\begin{proof}}{\end{proof}}

\numberwithin{equation}{section}

\newcommand{\R}{{\mathbb{R}}}
\newcommand{\C}{{\mathbb{C}}}

\newcommand{\Z}{{\mathbb{Z}}}

\newcommand{\Aa}{{\mathbf{A}}}
\newcommand{\Bb}{{\mathbf{B}}}
\newcommand{\Cc}{{\mathbf{C}}}
\newcommand{\Ee}{{\mathbf{E}}}
\newcommand{\Mm}{{\mathbf{M}}}
\newcommand{\Ll}{{\boldsymbol{\lambda}}}

\newcommand{\A}{{\mathcal{A}}}

\newcommand{\M}{{\mathcal{M}}}

\newcommand{\pa}{\partial}

\newcommand{\pr}{\operatorname{pr}}

\newcommand{\LCA}{LC\!\mathcal{A}}
\newcommand{\KCA}{KC\!\mathcal{A}}
\newcommand{\U}{S}
\newcommand{\longitude}{\operatorname{long}}
\newcommand{\mer}{\operatorname{mer}}
\newcommand{\selfl}{\operatorname{sl}}

\def\dfn#1{{\em #1}}

\title[Transverse Knot Contact Homology]{Filtrations on the knot
  contact homology \\ of transverse knots}

\author[T.~Ekholm]{Tobias Ekholm}
 \address{Uppsala University, Box 480,
  751 06 Uppsala, Sweden}
  \email{tobias@math.uu.se}

\author[J.~Etnyre]{John Etnyre}
\address{School of Mathematics, Georgia Institute of Technology,
Atlanta, GA 30332-0160}
\email{etnyre@math.gatech.edu}
\urladdr{\href{http://www.math.gatech.edu/~etnyre}{http://www.math.gatech.edu/\~{}etnyre}}

\author[L.~Ng]{Lenhard Ng}
\address{Mathematics Department, Duke University, Durham, NC 27708}
\email{ng@math.duke.edu}
\urladdr{\href{http://www.math.duke.edu/~ng/}{http://www.math.duke.edu/{}~ng/}}

\author[M.~Sullivan]{Michael Sullivan}
\address{Department of Mathematics, University of Massachusetts,
 Amherst, MA 01003-9305}
\email{sullivan@math.umass.edu}

\begin{document}

\begin{abstract}
We construct a new invariant of transverse links in the standard contact structure on $\R^3.$ This invariant is a doubly filtered version of the knot contact homology differential graded algebra (DGA) of the link, see \cite{EENS}, \cite{Ng08}. Here the knot contact homology of a link in $\R^3$ is the Legendrian contact homology DGA of its conormal lift into the unit cotangent bundle $S^*\R^3$ of $\R^3$, and the filtrations are constructed by counting intersections of the holomorphic disks of the  DGA differential with two conormal lifts of the contact structure. We also present a combinatorial formula for the filtered DGA in terms of braid representatives of transverse links and apply it to show that the new invariant is independent of previously known invariants of transverse links.
\end{abstract}


\maketitle

\section{Introduction} \label{sec:intro}
Constructing effective invariants of transverse knots in contact
3--manifolds that go beyond the obvious homotopy invariants  has been
a notoriously difficult problem. For knots transverse to the standard contact structure on $\R^3$, the first such invariant proven to be effective was
constructed only recently, using combinatorial knot Heegaard Floer
homology \cite{OzsvathSzaboThurston08}; a related transverse
invariant in the Heegaard Floer homology of a general compact contact
$3$--manifold was subsequently constructed in
\cite{LiscaOzsvathStipsiczSzabo09}. In this paper we define and
combinatorially present a new invariant for transverse links in the standard contact
$\R^3$. The new invariant is a filtration on the knot contact homology
differential graded algebra (DGA) \cite{EENS}, \cite{Ng08} for transverse
links and it appears to be quite different from the Heegaard Floer
invariants.  The construction of the invariant can be applied more
generally to produce invariants of transverse links in any contact
3--manifold and might even say interesting things in higher
dimensions, but the details are considerable more difficult, as are the
computations. So we defer the discussion of the general case to a
future paper.

\subsection{Transverse knot contact homology}\label{subsec:intro}
In \cite{Ng08}, the third author constructs a combinatorial DGA associated to a framed knot $K \subset \R^3$ and shows that its homology, combinatorial knot contact homology, is an invariant of framed knots. This invariant detects the unknot and encodes the Alexander polynomial, among other things.

The current authors prove in \cite{EENS} that the combinatorial DGA is (stable tame) isomorphic to the
Legendrian contact homology DGA of the conormal lift of a knot $K$ and also generalize the combinatorial description of knot contact homology to many-component links. Here the conormal lift of a link $K$ is a union $\Lambda_K$ of Legendrian tori in the unit cotangent bundle $S^{\ast}\R^{3}$ of $\R^{3}$ with the contact structure given as the kernel of the canonical  $1$-form. The Legendrian contact homology DGA of $\Lambda_K$ is an algebra generated by its Reeb chords with differential defined by a count of holomorphic disks in $S^{\ast}\R^{3}\times\R$ in the spirit of symplectic field theory \cite{EliashbergGiventalHofer00}. Its calculation in \cite{EENS} uses the contactomorphism $S^{\ast}\R^{3}\approx T^{\ast}S^{2}\times\R$ to transfer the holomorphic disk count in $S^{\ast}\R\times\R$ to $T^{\ast}S^{2}$. The disk count is then carried out using the relation between holomorphic disks and Morse flow trees \cite{Ekholm07}.

We will call the Legendrian invariant of $\Lambda_K$ defined using
holomorphic disks the \emph{knot contact homology} of $K$. We denote
the underlying DGA  by $(\KCA(K),\pa)$ and its homology by
$KCH(K)$. The algebra $\KCA(K)$ is a free graded
tensor algebra over the coefficient ring $R=\Z[H_1(\Lambda_K)]=\Z[\lambda_1^{\pm 1}, \mu^{\pm 1}_1,\dots,\lambda_r^{\pm1},\mu_r^{\pm 1}]$, where $r$ is the number of components of $K$; the differential $\pa$ depends on $K$ as well, but we suppress this dependence to simplify notation.

Let $(x_1,x_2,x_3)$ be coordinates on $\R^{3}$ and let $\xi_0=\ker(dx_3-x_1\,dx_2+x_2\,dx_1)$ denote the standard tight, rotationally symmetric
contact structure on $\R^3.$ 
For links $K$ transverse to $\xi_0$, we extend the coefficient ring of the knot contact homology DGA to $R[U,V],$ where $U$ and $V$ are two formal variables which encode intersections of holomorphic disks with the natural lifts of $\xi_0$ that correspond to its two coorientations. The resulting DGA will be denoted $(\KCA^-(K),\pa^-)$ and has a double filtration: by positivity of intersections the differential $\pa^-$ does not decrease the exponents of $U$ or $V$. Our main result is as follows.

\begin{thm}\label{thm:mainintr}
The filtered stable tame isomorphism type of $(\KCA^-(K),\pa^-),$ and
hence  its homology, is an invariant of the transverse knot or link $K$ in
$(\R^3,\xi_0).$ The DGA can be computed from a braid representative of $K$ as described in Theorem~\ref{thm:mainlinktv} below.
\end{thm}

This invariant of transverse links in particular distinguishes transverse knots with the same self-linking number and with the same Heegaard Floer invariants, see Section \ref{sec:HF}.
The explicit formula for $(\KCA^-(K),\pa^-)$ is somewhat involved and is therefore presented separately in Subsection \ref{subsec:computation}.
It should however be noted that our method for
obtaining the combinatorial formula relies heavily on the explicit description of holomorphic disks used in the
computation of knot contact homology in \cite{EENS}.
An alternate and purely combinatorial approach to Theorem \ref{thm:mainintr} is worked out in
\cite{Ng10}. It begins with the combinatorial description
of Theorem~\ref{thm:mainlinktv} in terms of a braid representation of the knot and proves invariance under
transverse isotopy, without any reference to the underlying geometry, via a study of effects of Markov moves. See \cite{Ng10} also for various algebraic properties of the invariant and more detailed calculations which demonstrate its effectiveness.

Those familiar with the algebraic formalism of Heegaard Floer homology
will notice that we can construct several other invariants from
$(\KCA^-(K),\pa^-).$ Two simpler invariants $(\widehat{\KCA}(K),\widehat\pa)$
and $\left(\widehat{\widehat{\KCA}}(K), \widehat{\widehat\pa}\right)$
 arise by
setting $(U,V)=(0,1)$ and $(0,0)$, respectively.  Setting $(U,V)=(1,1)$, on the other hand, reduces
$(\KCA^-(K),\pa^-)$ to the original knot contact
homology DGA, $(\KCA(K),\pa).$  Alternatively, one can tensor
$(\KCA^-(K),\pa^-)$ with an extended coefficient ring
$R[U^{\pm1},V^{\pm1}]$, to obtain a DGA,
$(\KCA^\infty(K),\pa^\infty).$ It turns out that
$(\KCA^\infty(K),\pa^\infty)$ depends on only the topological isotopy
class of $K$, rather than its transverse isotopy class.
\begin{thm}\label{thm:topNOTtran}
If $K$ is a transverse knot, then
as a DGA over
$\Z[\lambda^{\pm 1},\mu^{\pm 1},U^{\pm 1},V^{\pm 1}]$,
$(\KCA^\infty(K),\pa^\infty)$ is an invariant, up to stable tame
isomorphism, of the
topological knot underlying $K$.
\end{thm}
In fact,  it equals the Legendrian contact homology  of $\Lambda_K$ with coefficients in
a relative homology group, see Subsection~\ref{subsec:infty} and also
\cite{Ng10}. A statement similar to Theorem~\ref{thm:topNOTtran} holds
for general transverse links but will not be proven here.

\subsection{Calculating $(\KCA(K)^{-},\partial^-)$ from a braid presentation}\label{subsec:computation}
We now turn to the computation of our transverse link invariant. To this end we first need to introduce some notation. Our notation and conventions differ slightly from those of \cite{Ng10}, in which the equivalence of conventions is discussed.

The unit circle $U$ in the plane $\{x_3=0\}$ is transverse to the
contact structure defined by the $1$-form $dx_3-x_2\,dx_1+x_1\,dx_2$. By work
of Bennequin \cite{Bennequin83}, any transverse link $K$ can be isotoped
so that it is braided around $U$, i.e., contained in a tubular
neighborhood $U\times D^2$ of $U$ so that it is
transverse to the disk fibers. So $K$ can be represented as the
closure of a braid $B$ in $U\times D^2=S^1\times D^2$. The braid $B$
is an element of the braid group $B_n$ for some $n$ and can be
expressed as a word in the standard generators of $B_n$ and their inverses; 
here the standard generator $\sigma_j$
($j=1,\ldots,n-1$) corresponds to a braid that intertwines strands $j$ and $j+1$ positively.

The Legendrian contact homology DGA of $\Lambda_K$ is a
free associative non-commutative unital DGA
over the group ring of $H_1(\Lambda_{K};\Z)$.
Fixing a framing on each
component of $K$ (e.g., the standard topological $0$-framing) yields a distinguished basis
$\{\lambda_1,\mu_1,\ldots,\lambda_r,\mu_r\}$ of $H_1(\Lambda_{K};\Z)$,
where $r$ is the number of components of $K$. More precisely, by
identifying $T\R^3$ with $T^*\R^3$, we can
identify $\Lambda_K$ with the boundary of a tubular neighborhood of
$K$, which consists of $r$ disjoint tori. Each
component of $K$ has a natural orientation given by the positive coorientation
of the contact structure $\xi_0$. Then $\mu_j,\lambda_j$ are the
meridian and framed longitude of the $j$-th torus.

Now, as described above, suppose that $K$ is the closure of a braid $B\in B_n$ in a tubular
neighborhood of the standard transverse unknot $U.$
Let $\A_n$ denote the graded unital algebra over $\Z[\lambda_1^{\pm
  1},\mu_1^{\pm 1},\dots,\lambda_r^{\pm 1},\mu_r^{\pm 1}]$
freely generated by
\begin{align*}
&\{a_{ij}\}_{1\le i,j\le n;\,i\ne j} \text{ in degree }0,\\
&\{b_{ij}\}_{1\le i,j\le n;\,i\ne j} \text{ in degree }1,\\
&\{c_{ij}\}_{1\le i,j\le n} \text{ in degree $1$, and}\\
&\{e_{ij}\}_{1\le i,j\le n\,} \text{ in degree }2,
\end{align*}
where the degrees of $\lambda_j^{\pm 1}$ and $\mu_j^{\pm 1}$ equal $0$ for $j=1,\dots,r$. (When there is only one link component, we drop the subscripts on $\lambda_j,\mu_j$.)
Let $\A^{(0)}_n$ denote the subalgebra of $\A_n$ of elements of degree $0$.

We define an automorphism $\phi_B\colon \A^{(0)}_n\to
\A^{(0)}_n$ as follows. Introduce auxiliary variables
$\tilde{\mu}_1,\ldots,\tilde{\mu}_n$ of degree $0$, and write
$\tilde{\A}^{(0)}_n$
for the unital algebra over $\Z[\tilde{\mu}_1^{\pm
  1},\ldots,\tilde{\mu}_n^{\pm 1}]$ freely generated by
$\{a_{ij}\}_{1\leq i,j\leq n;i\neq j}$. For $k=1,\ldots,n-1$, let $\phi_{\sigma_k} : \tilde{\A}^{(0)}_n \to
\tilde{\A}^{(0)}_n$ be given by
\[
\begin{aligned}
\phi_{\sigma_k}(a_{ij})&=
a_{ij}  \quad & i,j\ne k, k+1\\
\phi_{\sigma_k}(a_{k\,k+1})&=
-a_{k+1\,k} \quad & \\
\phi_{\sigma_k}(a_{k+1\,k})&=
-\tilde{\mu}_k\tilde{\mu}_{k+1}^{-1}a_{k\,k+1}\quad &\\
\phi_{\sigma_k}(a_{i\,k+1})&=
a_{ik}\quad & i\not= k, k+1\\
\phi_{\sigma_k}(a_{k+1\, i})&=
a_{ki}\quad & i\not=k, k+1\\
\phi_{\sigma_k}(a_{ik})&=
a_{i\,k+1} - a_{ik}a_{k\, k+1}\quad& i < k\\
\phi_{\sigma_k}(a_{ik})&=
a_{i\,k+1} - \tilde{\mu}_{k}\tilde{\mu}_{k+1}^{-1}a_{ik}a_{k\,
  k+1}\quad& i > k+1\\
\phi_{\sigma_k}(a_{ki})&=
a_{k+1\,i} - a_{k+1\,k}a_{ki}\quad &i \ne k, k+1\\
\phi_{\sigma_k}(\tilde{\mu}_k^{\pm 1}) &= \tilde{\mu}_{k+1}^{\pm 1} \\
\phi_{\sigma_k}(\tilde{\mu}_{k+1}^{\pm 1}) &= \tilde{\mu}_k^{\pm 1} \\
\phi_{\sigma_k}(\tilde{\mu}_i^{\pm 1}) &= \tilde{\mu}_i^{\pm 1} \quad & i\neq k,k+1.
\end{aligned}
\]
Write $B$ in terms of braid group generators, $B=\sigma_{i_0}^{m_0}\ldots \sigma_{i_l}^{m_l}\in B_n,$ and let $\phi_B = (\phi_{\sigma_{i_0}})^{m_0}\circ
\cdots\circ(\phi_{\sigma_{i_l}})^{m_l}$. Then $\phi$ descends to a
homomorphism from $B_n$ to the automorphism group of
$\tilde{\A}^{(0)}_n$; in particular, $\phi$ satisfies the braid
relations.

For $j=i,\ldots,n$, let $\alpha(i)\in\{1,\ldots,r\}$ be the
number of the link component of $K$ corresponding to the $i$-th strand
of $B$. Then $\phi_B$ can be viewed as an automorphism of $\A^{(0)}_n$
by setting $\tilde{\mu}_i = \mu_{\alpha(i)}$ for all $i$ and having
$\phi_B$ act as the identity on $\lambda_i$ for all $i$. As an automorphism of
$\A^{(0)}_n$, $\phi_B$ acts as the identity on $\mu_i$ as well.

For convenient notation we assemble the generators of $\A_n$ into
$(n\times n)$-matrices.
Writing $\Mm_{ij}$ for the element in position $ij$ in the $(n\times n)$-matrix $\Mm$, we define the $(n\times n)$-matrices
\[
\begin{aligned}
\Aa:\quad &
\begin{cases}
\Aa_{ij}=a_{ij} &\text{ if }i<j,\\
\Aa_{ij}=\mu_{\alpha(j)} a_{ij} &\text{ if }i>j,\\
\Aa_{ii}=1+\mu_{\alpha(i)},
\end{cases}& \qquad\qquad
\Bb:\quad &
\begin{cases}
\Bb_{ij}=b_{ij} &\text{ if }i<j,\\
\Bb_{ij}=\mu_{\alpha(j)} b_{ij} &\text{ if }i>j,\\
\Bb_{ii}=0,
\end{cases}\\
\Cc:\quad &
\begin{cases}
\Cc_{ij}=c_{ij},
\end{cases}&
\Ee:\quad &
\begin{cases}
\Ee_{ij}=e_{ij}.
\end{cases}
\end{aligned}
\]
We also associate $(n\times n)$-matrices with coefficients in $\A_n^{(0)}$ to the braid $B$ as follows. Let $\phi_B(\Aa)$ be the matrix defined by $\bigl(\phi_B(\Aa)\bigr)_{ij}=\phi_B(\Aa_{ij})$. Then there are invertible matrices $\Phi^L_B$ and $\Phi^R_B$ so that
\[
\phi_B(\Aa)=\Phi^{L}_{B}\cdot \Aa\cdot \Phi^{R}_{B}.
\]
More specifically, we define these matrices by setting $B'$ to be the
$(n+1)$-braid obtained by adding an extra strand labeled $0$ to $B$
(that is, viewing the word defining $B$
as a word in the $(n+1)$-strand braid group generated by $\sigma_0,\ldots,\sigma_{n-1}$). Let $\phi'_B$ be the corresponding automorphism of the
$\{a_{ij}\}_{0\leq i,j\leq n;i\neq j}.$ Then define
$\Phi_B^L$ and $\Phi_B^R$ by
\[
\phi'_B (a_{i0})=\sum_{j=1}^n \left( \Phi_B^L\right)_{ij} a_{j0}\quad\text{and}\quad \phi'_B(a_{0j})=\sum_{i=1}^na_{0i}\left( \Phi^R_B\right)_{ij};
\]
see also \cite{EENS,Ng05,Ng08}.
(Note that since the $0$-th strand does not interact with the others,
$\tilde{\mu}_0$ does not appear anywhere in the expressions for
$\phi'_B(a_{i0})$ and $\phi'_B(a_{0j})$, and so $\Phi^L_B,\Phi^R_B$
have coefficients in $\A^{(0)}_n$.)

Also, define an $(n\times n)$ coefficient matrix $\Ll$ as
follows. Among the strands $1,\ldots,n$ of the braid, call a strand
\emph{leading} if it is the first strand belonging to its
component. Then let $\Ll$ be the diagonal matrix defined by
\[
\Ll:\quad
\begin{cases}
\Ll_{ij}=0 &\text{if }i\ne j,\\
\Ll_{ii}=\lambda_{\alpha(i)}\mu_{\gamma(i)}^{w(i)}
&\text{if the $i^{\rm th}$ strand is leading},\\
\Ll_{ii}= 1 &\text{otherwise},
\end{cases}
\]
where $w(i)$ is the writhe (algebraic crossing number) of the $i$-th component of the braid, considered by itself.
Finally, in order to capture the two filtrations, we define the following additional matrices:
\[
\begin{aligned}
\Aa^{U}:\quad &
\begin{cases}
\Aa_{ij}^{U}= U a_{ij} &\text{ if }i<j,\\
\Aa_{ij}^{U}=\mu_{\alpha(j)} a_{ij} &\text{ if }i>j,\\
\Aa_{ii}^{U}=U+\mu_{\alpha(i)},
\end{cases}\qquad\qquad&
\Bb^{U}:\quad &
\begin{cases}
\Bb_{ij}^{U}= U b_{ij} &\text{ if }i<j,\\
\Bb_{ij}^{U}=\mu_{\alpha(j)} b_{ij} &\text{ if }i>j,\\
\Bb_{ii}^{U}=0,
\end{cases}\\
\Aa^{V}:\quad &
\begin{cases}
\Aa_{ij}^{V}= a_{ij} &\text{ if }i<j,\\
\Aa_{ij}^{V}=\mu_{\alpha(j)} V a_{ij} &\text{ if }i>j,\\
\Aa_{ii}^{V}=1+\mu_{\alpha(i)} V,
\end{cases}&
\Bb^{V}:\quad &
\begin{cases}
\Bb_{ij}^{V}=b_{ij} &\text{ if }i<j,\\
\Bb_{ij}^{V}=\mu_{\alpha(j)} V b_{ij} &\text{ if }i>j,\\
\Bb_{ii}^{V}=0.
\end{cases}
\end{aligned}
\]
\begin{thm}\label{thm:mainlinktv}
The filtered knot contact homology DGA, $(\KCA(K)^{-},\pa^{-})$ of a
transverse link $K$ represented as a braid $B$ on $n$ strands is given
by the DGA, $(\A_n,\pa^-)$, with the differential $\pa^-\colon\A_n\to\A_n$ defined by the following matrix equations:
\begin{align*}
\pa^- \mathbf{A} &= 0, \\
\pa^- \mathbf{B} &= -\Ll^{-1}\cdot\Aa\cdot\Ll\,\, +\,\, \Phi^L_B \cdot \Aa \cdot \Phi^R_B, \\
\pa^- \mathbf{C} &= \Aa^{V}\cdot\Ll\,\, +\,\, \Aa^{U}\cdot\Phi^{R}_B, \\
\pa^- \mathbf{E} &= \Bb^{V}\cdot(\Phi_B^{R})^{-1}\,\, +\,\, \Bb^{U}\cdot\Ll^{-1}\,\, -\,\,
\Phi^L_B\cdot\Cc\cdot\Ll^{-1}\,\, +\,\, \Ll^{-1}\cdot\Cc \cdot (\Phi^R_B)^{-1},
\end{align*}
where if $\Mm$ is an $(n\times n)$-matrix, the matrix $\pa^- \Mm$ is defined by $(\pa^- \Mm)_{ij}=\pa^- (\Mm_{ij})$.
\end{thm}

The rest of the paper is organized as follows. We provide some general geometric background in
Section~\ref{sec:construction} before defining the transverse
invariant $(\KCA^-(K),\pa^-)$ and proving invariance in
Section~\ref{sec:definition}. In Section~\ref{sec:computations}, we
derive the combinatorial formula for $(\KCA^-(K),\pa^-)$ by proving Theorem~\ref{thm:mainlinktv}, and present examples
in Section~\ref{sec:ex}.

\section{Geometric Constructions}\label{sec:construction}
In this section we begin by recalling the definition of Legendrian contact homology and then discuss filtrations on the Legendrian contact homology DGA induced by complex hypersurfaces with certain properties. In the next subsection we recall the conormal construction and see how it can be used to construct invariants of smooth embeddings using Legendrian contact homology. In the last subsection we show how to construct an appropriate complex hypersurface in the case of the Legendrian contact homology of the conormal lift of a transverse knot in the standard contact structure on $\R^3$ and that it gives an invariant of transverse knots. We note that this section gives a geometric construction of transverse knots invariants that can be generalized to other situations, but to actually compute the invariant for transverse knots in the standard contact structure on $\R^3$ we will need to slightly alter the complex hypersurfaces so that they interact well with the constructions in \cite{EENS}. This is done explicitly in Section~\ref{sec:definition}.

\subsection{Legendrian contact homology}\label{lch}
In \cite{EliashbergGiventalHofer00}, the Legendrian contact homology
$LCH(\Lambda)$ of a Legendrian submanifold $\Lambda$ of a contact manifold
$(V,\xi)$ was introduced. The analytic underpinnings were worked out in detail in
\cite{EkholmEtnyreSullivan07} for a fairly general and useful case (but under the simplifying assumption that the chosen Reeb field of $\xi$ has no closed orbits, see below). In this case the Legendrian contact homology
$LCH(\Lambda)$ is the homology of a DGA, $(\LCA(\Lambda),\pa)$, over a fixed ring,
which changes by a particular type of quasi-isomorphism, called a stable
tame isomorphism, as $\Lambda$ changes by Legendrian isotopy. Thus, the
stable tame isomorphism class of $(\LCA(\Lambda),\pa)$ might be
considered to be the actual Legendrian invariant underlying $LCH(\Lambda)$.

We briefly sketch the definition of this DGA in the case handled in \cite{EkholmEtnyreSullivan07} for the convenience of
the reader and to establish some notation; for a more complete
definition, see \cite{EkholmEtnyreSullivan07} and for generalizations see \cite{EliashbergGiventalHofer00}.

Let $P$ be a manifold with exact symplectic form $d\lambda.$ The
manifold $P\times \R$ has a natural contact structure $\xi=\ker
(dz-\lambda)$ where $z$ is the coordinate on $\R.$ The Reeb vector field of this form is $\pa_z$ and consequently there are no closed Reeb orbits. Consider the projection

\[
\pi_\C:P\times\R \to P.
\]
The algebra $\LCA(\Lambda)$ is
the free tensor algebra generated over $\Z[H_1(\Lambda)]$ by the double points  of $\pi_\C(\Lambda).$ Notice that the double points are in one to one correspondence with ``Reeb chords", that is, flow lines of the Reeb vector field that begin and end on $\Lambda.$ Thus we will frequently refer to double points as Reeb chords.
For the double points we  choose ``capping paths'' in
$\Lambda$: that is, paths in $\Lambda$ that connect any Reeb chord endpoint to a fixed base point in its connected component, and fixed paths connecting the base points of distinct components; together, these give paths that connect the two points in $\Lambda$ which project to a double point in $\pi_\C(\Lambda).$
At a double point $p$ there are two points in $\Lambda$ that project to it. We label the one with larger $z$ coordinate $z^+$ and the other $z^-.$ The projection of a neighborhood of $z^+$ in $\Lambda$ to $P$ will be called the {\em upper sheet at $z$} and the projection of a neighborhood of $z^-$ will be called the {\em lower sheet at $z$}.
We then can define a Maslov type index $|c|$ and $|A|$ of Reeb chords $c$ and homology classes $A\in H_1(\Lambda)$ to provide a grading on $\LCA(\Lambda).$

To define the differential we fix an almost complex structure $J$ on $P$ (which can be thought of as an almost complex structure on $\xi$ using the isomorphism $d\pi_\C|_\xi$). For this almost complex structure the differential is determined by counting (pseudo-)holomorphic disks mapped into $P$ with boundary on $\pi_\C({\Lambda}).$
Given a Reeb chord $a,$ a (noncommutative) word of
Reeb chords ${\mathbf b} = 
b_1\cdots b_m,$ and
a homology class $A \in  H_1(\Lambda),$
we define a moduli space
\begin{equation}
\label{eq:projmdli}
{\M}_A(a;\mathbf{b})
\end{equation}
of holomorphic disks $u\colon D\to P$ with: boundary on $\Lambda;$ one positive\footnote{Using the orientation on $\partial D$ induced by the complex structure a puncture is positive if it maps the segment of the boundary just before the puncture to the lower sheet at the double point and the segment just after the puncture to the upper sheet. The puncture is negative if the roles of the upper and lower sheets are reversed.} puncture at $a$ and negative punctures at $b_1,\dots,b_m$ in the order given by the boundary orientation; and the homology class $A$ given by the lift of $u(\partial D)$ to $\Lambda$ together with the chosen capping paths.
For a generic almost complex structure, this moduli space is a manifold of dimension $|a| - |\mathbf{b}|  - |A|-1$, where $|\mathbf{b}|=\sum_{i=1}^{m} |b_i|$. Furthermore, the space has a natural compactification which consists of (several level) broken curves and which admits the structure of a manifold with boundary with corners.   The moduli spaces can be coherently oriented provided the Legendrian submanifold $\Lambda$ is spin.

Define the differential on the generators of $\LCA(\Lambda)$ by
\begin{equation}\label{stddiff}
        \partial a =
        \sum_{\{u \in {\M}_A(a;{\mathbf b}) \,\,|\,\,
        |a|-|\mathbf{b}|-|A| -1=0\}}
        (-1)^{|a|+1}
        \sigma(u)  e^A {\mathbf b},
\end{equation}
where $\sigma(u) \in \{\pm1\}$ is determined by the moduli
space orientation. The differential is then extended to all of
$\LCA(L)$ by the graded product rule and linearity.

\subsection{Almost complex hypersurfaces and filtration on Legendrian contact homology}
\label{subsec:filtration}
We discuss how to use a complex hypersurface to add a ``filtration'' to the Legendrian contact homology DGA.
With the notation above, suppose that $\overline{H}$ is a submanifold of $P$ such that
\begin{itemize}
\item[$(1)$] $\overline{H}$ is an almost complex hypersurface ($J(T\overline{H})=T\overline{H}$ and the (real) codimension of $\overline{H}$ equals $2$) and
\item[$(2)$] $\Lambda\cap \overline{H}=\varnothing.$
\end{itemize}
Given such $\Lambda$ and $\overline H$, we can extend the
base ring for the contact homology DGA of $\Lambda$ from $\Z[H_1(\Lambda)]$
to $\Z[H_1(\Lambda)][U_{\overline H}]$, by adjoining a formal variable $U_{\overline H}$ and
changing the definition of the boundary map, using powers of $U_{\overline H}$ to keep
track of the number of times the holomorphic disks in the
definition of the boundary map intersect $U_{\overline H}.$ Specifically, given
$u\in  \mathcal{M}_A(a;{\mathbf b}),$ positivity of intersection
shows that the intersection number of the image of $u$ with ${\overline H}$ is a
well-defined and nonnegative integer which we denote $n_{\overline H}(u).$ We can
now modify the differential in Equation~\eqref{stddiff} and define instead:
\begin{equation}\label{filtdiff1}
        \partial_f a =
        \sum_{\{u \in \mathcal{M}_A(a;{\mathbf b}) \,\,|\,\,
        |a|-|\mathbf{b}|-|A| -1=0\}}
        (-1)^{|a|+1}
        \sigma(u)   U_{\overline H}^{n_{\overline H}(u)}e^A {\mathbf b}.
\end{equation}
Conditions $(1)$ and $(2)$ above ensure that $\partial_f$ is a filtered
differential; that is, it respects the filtration
\[
\LCA \supset U_{\overline H} \cdot \LCA \supset U_{\overline H}^2 \cdot \LCA \supset \cdots.
\]
The proof that $(\LCA(\Lambda),\partial)$ is invariant up to stable tame isomorphism carries over to show that the
stable tame isomorphism type of $(\LCA(\Lambda),\partial_f)$ over $\Z[H_1(\Lambda)][U_{\overline H}]$ is an
invariant of $\Lambda$ under isotopies $\Lambda_t$, $0\le t\le 1$, such that $\Lambda_t$ satisfies condition $(2)$ for all $t$, see
Theorem \ref{thm:filterinv} below.

\subsection{The conormal construction and knot contact homology}
\label{subsec:conormal}
Given any $n$-manifold $M,$ the cotangent bundle $T^*M$ has a
canonical symplectic structure $d\lambda$ where $\lambda$ is the
Liouville 1--form. If we choose a metric $g$ on $M$ then we can
consider the unit cotangent bundle $\U^*M.$ The restriction of
$\lambda$ to $\U^*M$ is a contact form which we denote by $\alpha$, and
$\xi=\ker \alpha$ is a contact structure on $\U^*M.$

Let $\pr\colon T^{\ast} M\to M$ denote the natural projection. If $K$ is a
submanifold of $M$ (of any dimension) then the unit conormal bundle
\[
\Lambda_K= \{\beta\in \U^*M: \pr(\beta)=p\in K \textrm{ and  }  \beta(v)=0 \textrm{ for all } v\in T_pK \}
\]
is a Legendrian submanifold of $(\U^*M,\xi).$ If we smoothly isotop $K$
in $M$, it is clear that $\Lambda_K$ will undergo a Legendrian isotopy
in $(\U^*M,\xi).$ Thus any Legendrian isotopy invariant of $\Lambda_K$
is a smooth isotopy invariant of $K.$

In this paper, we consider conormal lifts of links $K\subset \R^{3}$ which are Legendrian submanifolds $\Lambda_K\subset S^{\ast}\R^{3}$. There is a well known contactomorphism $S^{\ast}\R^{3}\cong J^{1}(S^{2})=T^{\ast}S^{2}\times\R$ and we will thus consider $\Lambda_K\subset T^{\ast} S^{2}\times\R$ and use the version of Legendrian contact homology of $\Lambda_K$
defined in \cite{EkholmEtnyreSullivan07}. In particular, we define the \dfn{knot contact homology algebra} of a link $K$ in $\R^3$ to be the Legendrian contact homology algebra of $\Lambda_K  \subset J^1(S^2)$. We denote it $(\KCA(K),\pa)$ and note that the stable tame isomorphism class of $(\KCA(K),\pa)$ is an isotopy invariant of $K$. In \cite{EENS} the authors show how to compute $(\KCA(K),\partial)$ and demonstrate that it is equivalent to the combinatorial knot DGA introduced by the third author in \cite{Ng08}.

We recall for later use that the projection $\pi_F:J^1(S^2)\to S^2\times \R$ is called the front projection and that a generic Legendrian submanifold $\Lambda$ in $J^1(S^2)$ can be recovered uniquely from $\pi_F(\Lambda)\subset S^2\times\R.$ Since $S^2\times \R$ can be visualized as $\R^3\setminus \{(0,0,0)\},$ it will frequently be useful to study a Legendrian submanifold $\Lambda$ via its front projection.

\subsection{Transverse link invariants}

Here we describe the geometry underlying the claimed filtration on knot contact homology, in the case when $K \subset \R^3$ is a transverse link and not simply a topological link.
Consider a contact structure $\xi$ on $\R^3.$ If there is a Reeb vector field $R_\xi$ for $\xi$ such that the flow lines of the vector field trace out geodesics (in the flat Euclidean metric on $\R^3$) then we can consider the submanifolds
\[
H_\xi^\pm=\{\eta\in S^*\R^3: \pm\eta(R_\xi)>0 \text{ and } \eta(v)=0 \text{ for all } v\in \xi\}.
\]
Since the Reeb flow lines of $R_\xi$ are geodesics, $H_\xi^\pm$ are foliated by Reeb flow lines in $S^*\R^3.$ In other words, identifying $S^*\R^3$ with $J^1(S^2)=T^*S^2\times \R$, the projected submanifold $\overline{H}_\xi^\pm=\pi_\C(H_\xi^\pm)$ in $T^*S^2$ is  an embedded codimension 2 submanifold. One may also choose the almost complex structure on $T^*S^2$ so that $\overline{H}_\xi^\pm$ is a holomorphic submanifold. Moreover, if $K$ is a knot in $\R^3$ that is transverse to $\xi$ then its conormal lift $\Lambda_K$ projects to an exact Lagrangian submanifold in $T^*S^2$ that is disjoint from $\overline{H}_\xi^\pm.$ Thus, as discussed above, we can construct an associated filtered contact homology DGA that will be an invariant of the transverse isotopy class of $K.$

To carry out the above construction one must choose
a contact form $\alpha$ with $\xi = \ker(\alpha)$
so that its Reeb flow traces out geodesics. The standard contact structure does have such representatives, for example $\xi_0=\ker(\sin x_1\, dx_2 + \cos x_2 \, dx_3),$ but it is quite difficult to actually compute the filtered contact homology DGA for this contact structure. To take advantage of the computations in \cite{EENS} we would prefer to work with the contact form $\ker(dx_3+r^2\, d\theta),$ but this contact form does not have a Reeb vector field with the requisite properties. In the remainder of the paper we overcome this problem by considering a different $S^2$-subbundle $B$ of $T^*\R^3$ instead of the unit cotangent bundle. As long as each fiber in this subbundle bounds a convex region containing the origin, we can still identify $B$ with $J^1(S^2).$ By a judicious  choice of $B$ we will see that the projection of $H_\xi^\pm$ to $T^*S^2$ retains enough of the properties discussed above to allow us to explicitly calculate a filtered invariant for transverse knots in $\R^3.$

\section{The Filtered DGA of Transverse Links in $(\R^{3},\xi_0)$} \label{sec:definition}

In this section we show how to construct from the standard contact structure $(\R^{3}, \xi_{0})$ a pair of complex hypersurfaces $\overline{H}_{\pm}$ in $T^{\ast} S^{2}$ satisfying the following:  if $K$ is any link in a sufficiently small ball around the origin which is transverse to $\xi_0$, then $\overline{H}_{\pm}\cap\pi_\C(\Lambda_K)=\varnothing$.
In order to get a computable invariant we adapt the geometry and use a slightly non-standard version of $S^{\ast}\R^{3}$.

In Subsection~\ref{defS} we describe our geometric model of $S^{\ast}\R^{3}$ and its relation to $T^{\ast} S^{2}\times\R$. In Subsection~\ref{subsec:confining} we show that we can control the image of the holomorphic disks used to compute the knot contact homology so that they lie near the zero section in
$T^{\ast} S^2$ if the original link is sufficiently small.  We then discuss the conormal lift of the standard contact structure on $\R^3$ in Subsection~\ref{subsec:lift} and show in Subsection~\ref{transiso} that it is no restriction to assume that all transverse links and isotopies lie in a small ball around the origin. In Subsections~\ref{ssec:acs} and~\ref{ssec:filteredDGA}, we define a suitable almost complex structure on $T^{\ast} S^2$ and then prove the filtered DGA of a transverse link is well-defined and invariant up to stable tame isomorphisms under isotopies through transverse links. Finally, in Subsection~\ref{subsec:infty}, we explain why the infinity theory for knots in $\R^3$ is a topological invariant.

\subsection{The spherical conormal bundle and the $1$-jet space of $S$}\label{defS}
Let $y=(y_1,y_2,y_3)$ be standard Euclidean coordinates on
$\R^{3}$. Let $S$ denote the smooth boundary of a (not necessarily
strictly) convex subset of $\R^{3}$ which is symmetric with respect to
reflection in the $y_1y_2$-plane and with respect to rotations about
the $y_3$-axis. For $y\in S$, let $\nu(y)$ denote the outward unit
normal to $S$ at $y$. Note that the symmetries of $S$ imply that
$\nu(y_1,y_2,y_3)$ lies in the subspace spanned by the vectors
$(y_1,y_2,0)$ and $(0,0,y_3)$. In particular, $\nu(0,0,y_3)=(0,0,\pm
1)$ and $\nu(y_1,y_2,0)=\frac{1}{\sqrt{y_1^{2}+y_2^{2}}}(y_1,y_2,0)$.

We represent the spherical cotangent bundle of $\R^{3}$ as
\[
\U^*\R^{3}=\R^{3}\times S\subset T^{\ast}\R^3=\R^3\times\R^{3}
\]
and use coordinates $(x,y)=(x_1,x_2,x_3,y_1,y_2,y_3)$ on $T^{\ast}\R^{3}$. The contact form on $\U^*\R^{3}$
is the restriction of the Liouville $1$--form $y\cdot dx=\sum_{j=1}^{3} y_j\,dx_j$ on $T^{\ast}\R^{3}$ to $\U^*\R^{3}$.
We compute the Reeb vector field as follows.
\begin{lma}\label{reeb}
The Reeb vector field on $\U^*\R^{3}$ is
\begin{equation}\label{eq:ReebS}
R(x,y)=|y|^{-1}(\nu(y)\cdot\pa_x)=|y|^{-1}\sum \nu_j(y)\pa_{x_j}
\end{equation}
and the time $t$ flow starting at $(x,y)$ is
\begin{equation}\label{eq:ReebSflow}
\Phi_R^{t}(x,y)=\left(x+t\left(|y|^{-1}\nu(y)\right)\,,\,y\right).
\end{equation}	
\end{lma}
\begin{pf}
If $i$ denotes the standard complex structure on $\C^{3}=\R^{3}+i\R^{3}=T^{\ast}\R^{3}$ then the Reeb field lies in the intersection of $T(\U^*\R^3)$ and the complex tangent line at $y$ containing $\nu(y)$. Thus up to normalization the Reeb field equals $-i\nu(y)=\nu(y)\cdot\pa_{x}$. The lemma follows.
\end{pf}

The spherical cotangent bundle $\U^* \R^3$ can be identified with the 1--jet space $J^1(S)=T^*S\times \R$ as follows, where we use the flat metric on $\R^{3}$ to identify vectors and covectors.
\begin{lma}\label{lma:contactomorphism}
The map $\phi\colon \U^*\R^{3}\to J^{1}(S)=T^{\ast}S\times\R$ given by
 \[
\phi(x,y)=(y,x-(x\cdot\nu)\nu,x\cdot y).
\]
is a contactomorphism
\[
\phi\colon (\U^*\R^{3},y\cdot dx)\to (J^{1}(S),dz-p\cdot dq),
\]
where $q=(q_1,q_2)$ are local coordinates on $S,$ $p=(p_1,p_2)$ give the coordinates on the fiber of the cotangent bundle and $z$ is the coordinate on $\R.$
\end{lma}

\begin{pf}
Note that
\begin{align*}
\phi^{\ast}(dq)&=dy,\\
\phi^{\ast}(dz)&=x\cdot dy + y\cdot dx,
\end{align*}
and thus
\[
\phi^{\ast}(dz-p\cdot dq)=x\cdot dy +y\cdot dx - (x-(x\cdot\nu)\nu)\cdot dy= y\cdot dx,
\]
where we use $\nu\cdot dy=0$, which holds since $y\in S$ and $\nu$ is the normal of $S$ at $y$.
\end{pf}

\subsection{Confining holomorphic curves}
\label{subsec:confining}
Fix $\delta>0,$ let $S\subset\R^{3}$ be as above and write $\rho_0=\max\{|y|\colon y\in S\}$. Below we will measure lengths of cotangent vectors in $T^{\ast}S$ using the metric coming from the one induced on $T^{\ast}\R^{3}$ by the flat metric on $\R^{3}$. If $(q,p)\in T^{\ast}S$ then we write $|p|$ for the length of the cotangent vector measured with respect to this metric.

\begin{lma}\label{lma:smallaction}
If $K$ is any link contained in $B(\delta),$ the ball of radius $\delta$ about the origin in $\R^3,$ then
\[
\Lambda_K\subset\{(q,p,z)\in T^{\ast}S\times\R\colon |p|\le\delta\}.
\]
Moreover, if $c$ is a Reeb chord of $\Lambda_K$ then
\[
\int_c (dz-p\,dq)\le 2\rho_0\delta.
\]
\end{lma}
\begin{pf}
By Lemma \ref{lma:contactomorphism} if $x\in K$ and $(x,y)\in \Lambda_K\subset \U^*\R^{3}$ then
\[
|p|=|x-(\nu(y)\cdot x)\nu(y)|\le|x|\le\delta.
\]
For the second statement we note that Lemma~\ref{reeb} implies that a Reeb chord $c$ of $\Lambda_K$ is a lift $(l_c,y)$ into $\U^*\R^{3}$ of a line segment $l_c$ in $\R^3$ with endpoints on $K$ and that the action of the chord is
\[
\int_{l_c} y\cdot dx\le \int_{l_c}\rho_0|dx|\le 2\rho_0\delta.
\]
\end{pf}

As a consequence of Lemma \ref{lma:smallaction} we can confine holomorphic curves with boundary on $\Lambda_K.$
As in Subsection \ref{subsec:conormal}, $\pi_\C:J^1(S)\to T^*S$ will denote the projection map and
$\overline\Lambda_K=\pi_\C(\Lambda_K)$.

\begin{lma}\label{lma:confine}
Let $J$ be an almost complex structure on $T^*S$ that is compatible with the symplectic form on $T^*S.$ Fix $\delta_0>0$ a constant. Then there exists $0<\delta<\delta_0$ such that if $K$ is a link in $B(\delta)$ then any $J$-holomorphic disk with boundary on $\overline\Lambda_K$ and one positive puncture lies in $\{(q,p)\in T^{\ast}S\colon |p|<2\delta_0\}$.
\end{lma}

\begin{pf}
Consider $0<\delta<\delta_0$. By Stokes' theorem and
Lemma \ref{lma:smallaction}, the area of a disk $u$ as described is bounded above by $2\rho_0\delta$ and its boundary is contained in the region where $|p|\le \delta$. By monotonicity (see for example Proposition 4.3.1
\cite{AudinLafontaine94}), there exists a constant $C_0$ (depending only on $J$) such that if $u$ leaves the region where $|p|\le 2\delta_0$ then the area of $u$ is at least $C_0\delta_0^{2}$. If we now take $\delta< C_0\delta^{2}_0/2\rho_0$, the lemma follows.
\end{pf}

\subsection{A contact form on $\R^{3}$ and its spherical cotangent lifts}
\label{subsec:lift}
Fix the contact form $\alpha_0=dx_3-x_2\,dx_1+x_1\,dx_2$ on $\R^{3}$ and write, as in Section \ref{sec:intro}, $\xi_0=\ker(\alpha_0)$ for the corresponding contact structure. Note that $\alpha_0$ is invariant under rotations in the $x_1x_2$-plane and that the diffeomorphism
\[
(x_1,x_2,x_3)\mapsto\left(\tfrac{1}{2}x_1\,,\,\tfrac{1}{2}x_2\,,\,x_3-\tfrac{1}{2}	 x_1x_2\right)
\]
gives a contactomorphism between $\alpha_0$ and the standard contact form on $\R^{3}$, $dx_3-x_2\,dx_1$.

If $v\in\R^{3}$ is a non-zero vector, we denote the two open half rays determined by $v$ as follows:
\[
\R_\pm\cdot v=\{x\in\R^{3}\colon x=tv,\,\,\pm t>0\}.
\]
The \emph{positive and negative spherical lifts} of $\alpha_0$ are
\begin{equation}\label{eq:sphlift}
H_{+}=\left\{(x,y)\in \U^*\R^{3}\colon y\in \R_+\cdot\alpha_0(x)\right\}
\text{ and }
H_{-}=\left\{(x,y)\in \U^*\R^{3}\colon y\in \R_-\cdot\alpha_0(x)\right\},
\end{equation}
respectively. At all points $x$ on a transverse link $K,$
$\alpha_0$ fails to annihilate the tangent space $T_xK$.
Thus we have the following immediate result.
\begin{lma}\label{lma:disjoint1}
If $K$ is transverse to $\xi_0$ then the conormal lift $\Lambda_K$ of $K$ is disjoint
from $H_{\pm}.$
\end{lma}

\subsection{Shrinking transverse links}\label{transiso}
The following straightforward lemma reduces the study of transverse links in $\R^{3}$ with its standard contact structure to the study of  transverse links lying in an arbitrary small fixed neighborhood of the origin.
Let $B^{d}(r)$ denote the
closed $d$-dimensional ball of radius $r$ around the origin
and let $B(r) = B^3(r).$

\begin{lma}\label{lma:shrink}
Fix $\delta >0.$
Let $K(s)$, $s\in B^{d}(1),$ be a continuous family of transverse  links in $(\R^{3},\xi_0)$ such that $K(s)\subset B(\delta)$ for $s\in\pa B^{d}(1).$
Then there is a homotopy $K(s,t)$, $0\le t\le 1$, with $K(s,t)=K(s)$
if $s\in\pa B^{d}(1)$, and such that $K(s,0)=K(s)$ and $K(s,1)\subset
B(\delta)$ for each $s\in B^{d}(1)$.
In particular, the space of transverse  links in
$\R^{3}$ is weakly homotopy equivalent to the space of transverse
links in $B(\delta)$.
\end{lma}

\begin{pf}
Note that if $K$ is a transverse link in
$(\R^{3},\alpha)$ then so is $F_c(K)$ where
\[
F_c(x_1,x_2,x_3)=(\sqrt{c}\,x_1,\sqrt{c}\,x_2,c x_3)
\]
for $c>0.$
Choose $\epsilon_1, \epsilon_2 >0$ sufficiently small so that
$K(s) \subset B(\epsilon_1^{-1} \delta)$ for all $s \in B^d(1)$
and
$K(s) \subset B(\delta)$ for all
$s \in (B^d(1) \setminus B^d(1-\epsilon_2)).$
Choose any smooth
$\varepsilon: B^d(1) \rightarrow [\epsilon_1, 1]$
such that
\[
\varepsilon(s) = 1 \,\,\, \mbox{if} \,\, s \in \pa B^d(1),
\quad \mbox{and} \,\,\,
\varepsilon(s) = \epsilon_1 \,\,\, \mbox{if} \,\, s \in B^d(1 - \epsilon_2).
\]
Then $K(s,t)=F_{(1-t)+t\varepsilon(s)}(K(s))$ is a homotopy with the required properties.
\end{pf}
Using Lemma \ref{lma:shrink} in conjunction with Lemma~\ref{lma:confine}, we
can restrict our attention to a small neighborhood of the zero section
in $T^*S$ when counting holomorphic curves with boundary on the
conormal lift of transverse knots.

\subsection{Almost complex structures}
\label{ssec:acs}
We choose  $S\subset \R^{3}$ as in Subsection~\ref{defS} with the additional requirement that $S$ is flat near the north and south poles. More precisely, for some fixed $\delta_0>0$ we require
that
\begin{align}\label{eq:shapeofS}
S \cap\left\{y\in\R^{3}\colon \sqrt{y_1^{2}+y_2^{2}}\le 3\delta_0\right\}
=\left\{(y_1,y_2,\pm 1)\in\R^{3}\colon \sqrt{y_1^{2}+y_2^{2}}\le 3\delta_0\right\}.
\end{align}
For $0 < k <3$, write
\begin{align*}
E_{k\delta_0}=
S \cap\left\{y\in\R^{3}\colon \sqrt{y_1^{2}+y_2^{2}}\le k\delta_0\right\}
=\left\{(y_1,y_2,\pm 1)\in\R^{3}\colon \sqrt{y_1^{2}+y_2^{2}}\le k\delta_0\right\},
\end{align*}
and let $H_{\pm}(k\delta_0)$ denote the intersection $H_{\pm}\cap (T^{\ast}E_{k\delta_0} \times\R).$ Note that the metric on $S$ is flat in $E_{3\delta_0}$ and that the almost complex structure induced by the metric agrees with the standard integrable complex structure $J_0$ on $T^{\ast}E_{3\delta_0}\subset\C^{2}$.

Let $K$ be a transverse link in the ball $B(\delta)$ and $\Lambda_K$ its conormal lift. As usual, let $\pi_\C\colon J^{1}(S)\to T^{\ast}S$ denote the projection, and write
$\overline{H}_\pm = \pi_\C({H}_\pm)$ and $\overline{H}_\pm(k\delta_0) = \pi_\C(H_{\pm}(k\delta_0)).$

\begin{lma}\label{lma:subvariety}
The spherical lifts $H_{\pm}$ of $\alpha_0$ intersect the subset $\pa(T^{\ast}E_{2\delta_0}\times\R)\subset J^{1}(S)$ transversely. Moreover, $H_{\pm}(2\delta_0)$ is invariant under the Reeb flow and its projection
${\overline H}_{\pm}(2\delta_0) \subset T^{\ast}E_{2\delta_0}$ is a smooth $J_0$-complex subvariety.
\end{lma}

\begin{pf}
By Formula~\eqref{eq:shapeofS}, the normal vector to $E_{2\delta_0}$ is $\nu = (0,0,1).$
By the definition of the contact form $\alpha_0$ and
Lemma~\ref{lma:contactomorphism},
\begin{equation}\label{eq:localH}
\begin{aligned}
H_{\pm}(2 \delta_0)  =
\bigl\{ (q,p,z)\in T^{\ast}S\times\R\colon &  (q,p,z)=\left(\pm(-x_2,x_1,1)\,,\,(x_1,x_2,0)\,,\,x_3\right)\\
& \mbox{ and }x_1^{2}+x_2^{2}\le 4\delta_0^2\, \bigr\}
\end{aligned}
\end{equation}
(see the proof of Lemma \ref{lma:disjoint2} for a parameterized version).
This is clearly transverse to $\pa(T^{\ast}E_{2\delta_0}\times\R)$ and invariant under the Reeb flow (which is just translation in the $z$-direction). Furthermore, under the identification $((u_1,u_2,\pm 1),(v_1,v_2,0))\mapsto(u_1+iv_1,u_2+iv_2)\in\C^{2}$, $\overline{H}_{\pm}(2 \delta_0)$ corresponds to the complex line
\[
(u_1+iv_1,u_2+iv_2)=(i\zeta,\zeta),
\]
where $\zeta=x_1+ix_2$.
\end{pf}

Fix $\delta\in (0,\delta_0)$ so that Lemma \ref{lma:confine} holds. As pointed out in Subsection \ref{transiso},  Lemma
\ref{lma:shrink} implies that when studying the isotopy classification of links in
$\R^{3}$ which are transverse to $\alpha$, it is no restriction to
assume that all such links are contained in  $B(\delta)$ and that all
isotopies are through links inside $B(\delta)$. We thus make this
assumption throughout the rest of the paper.

\begin{lma}  \label{lma:disjoint2}
Let $K$ be a transverse link (by our standing assumption $K\subset B(\delta)$). Then
the sets $\overline{\Lambda}_K$ and $\overline{H}_{\pm}$ are disjoint in $T^{\ast}S$.
In addition,
$\overline{H}_{\pm} \setminus \overline{H}_{\pm}(2 \delta_0)$
does not intersect any  $J$-holomorphic disk with boundary on
$\overline{\Lambda}_K$ and one positive puncture.
\end{lma}

\begin{proof}
Lemmas \ref{lma:disjoint1} and  \ref{lma:subvariety} imply that
$\overline{H}_\pm(2 \delta_0)$ and $\overline{\Lambda}_K$
are disjoint, for if $H_\pm(2 \delta_0)$ intersected a Reeb flow line emanating from
$\Lambda_K,$ then it would intersect $\Lambda_K$
itself since ${H}^\pm(2 \delta_0)$
is foliated by Reeb flow lines.

Lemmas~\ref{lma:smallaction} and~\ref{lma:confine} imply that any holomorphic curve with boundary on $\overline\Lambda_K$ will be contained in
\[
\{(q,p)\in T^{\ast}S\colon |p|<2\delta_0\}.
\]
Write $x=(x_1,x_2,x_3).$
The contactomorphism of Lemma \ref{lma:contactomorphism} and the properties of $S$, see Subsection~\ref{defS}, imply that $\overline{H}_\pm$ consists of the points $(p(x),q(x))\in T^{\ast} S$
whose coordinates satisfy the following:
\[
q(x)=b(-x_2,x_1,1),
\]
where $b>0$ is such that $q(x)\in S;$ and
\[
p(x)=\left(x_1-x_2x_3\sqrt{1-a^2(x_1^2+x_2^2)}\,\,,\,\, x_2-x_1x_3\sqrt{1-a^2(x_1^2+x_2^2)}\,\,,\,\,  a^2x_3(x_1^2+x_2^2)\right),
\]
where $a\geq 0$ is chosen so that $\nu(q(x))= (-ax_2, ax_1, \sqrt{1-a^2(x_1^2+x_2^2)})$.  It follows that
\[
|p(x)|^{2}=(x_1^2+x_2^2)(1+a^2x_3^2).
\]
Since the second factor is bigger than or equal to 1, the lemma clearly follows.
\end{proof}

\begin{lma}\label{lma:tv1}
There exists an almost complex structure $J$ on $T^{\ast} S$ which agrees with $J_0$ in a neighborhood of $\overline{ H}_{\pm}$ and outside the region where $|p|\le\delta_0,$ and which is regular for $\overline{\Lambda}_K$ in the sense that $0$- and $1$-dimensional moduli spaces of holomorphic disks with boundary on $\Lambda_K$ and one positive puncture are transversely cut out.
\end{lma}
\begin{pf}
Proposition  2.3(1) in \cite{EkholmEtnyreSullivan07} shows that the asserted regularity can be achieved by perturbing $J$ in an arbitrary small neighborhood of the double points of $\overline\Lambda_K$. Since these double points lie neither on $\overline{H}_{\pm}$ nor in the region where $|p|\ge\delta_0$ the lemma follows.
\end{pf}

\subsection{A filtered DGA}
\label{ssec:filteredDGA}
Following the discussion in Section~\ref{sec:construction},
we now define a filtered version of
the Legendrian contact homology DGA
of $\Lambda_K$ when $K$ is a transverse link.
See
\cite{EkholmEtnyreSullivan07} for further background details
on the unfiltered DGA.

Let $\KCA^-(K) = \LCA(\Lambda_K)$ be the graded free associative
non-commutative unital algebra over the ring
$\Z[H_1(\Lambda_K)][U, V]$ generated by the
Reeb chords of $\Lambda_K$.
Here $U,V$ are two (formal) variables of grading 0. Other generators and coefficients have grading exactly as in the usual Legendrian contact homology DGA determined via a Maslov index. We denote the grading $|\cdot|$.

Consider a Reeb chord $a,$ a monomial of
Reeb chords ${\mathbf b} =
b_1\cdots b_m,$ and
a homology class $A \in  H_1(\Lambda_K).$
Recall from \eqref{eq:projmdli} the moduli space ${\M}_A(a;{\mathbf b})$ of holomorphic disks with boundary on $\Lambda_K$, which has dimension $|a|-|\mathbf{b}|-|A|-1$. For $u \in {\M}_A(a;{\mathbf b}),$ let $n_U(u)$
and $n_V(u)$ denote the algebraic intersection of
$u$ and $\overline{H}_+$ and $u$ and $\overline{H}_-,$
respectively. Lemmas \ref{lma:subvariety} and \ref{lma:disjoint2}
imply that these counts are well-defined and non-negative for the $J$ and $K$ we consider.

Define the differential ${\partial}^- \colon \KCA^-(K) \to\KCA^-(K)$ by
\begin{equation}
\label{FilteredDifferential.eq}
       {\partial}^- a =
        \sum_{\{u \in {\M}_A(a;{\mathbf b}) \,\,|\,\,
        |a|-|\mathbf{b}|-|A|=1\}}
        (-1)^{|a|+1}
        \sigma(u) \,U^{n_U(u)} V^{n_V(u)}\, e^A\, {\mathbf b},
\end{equation}
where $\sigma(u) \in \{\pm1\}$ is determined by the moduli
space orientation; compare to Equation~\eqref{filtdiff1}.
Setting $U = V=1$ we recover the differential used
in the Legendrian contact homology DGA defined in \cite{EkholmEtnyreSullivan07}.

\begin{thm}\label{thm:filterdiff}
The above definition gives a filtered differential: $\pa^{-}$ does not decrease the exponents of $U$ or $V,$ and $(\pa^-)^2=0$.
\end{thm}

\begin{pf}
The unfiltered version of the differential squares to zero due to the usual transversality, gluing and compactness arguments \cite{EkholmEtnyreSullivan07}.
The filtered version follows from the fact that $n_U$ and $n_V$
are nonnegative, the fact  that gluing and compactness respects the
filtration, and Lemmas~\ref{lma:disjoint2} and~\ref{lma:tv1}.
\end{pf}

We call the above filtered DGA the \dfn{transverse knot DGA of $K$}
and denote it $(\KCA^-(K),\pa^-)$.
We next show that the filtered DGA of $K$ is invariant under
transverse isotopies of $K$ up to filtered stable tame isomorphism.
(For the definition of stable tame isomorphism, which extends to our
situation, see \cite{EkholmEtnyreSullivan07} for example.)
In particular, the homology of the filtered DGA  $(\KCA^-(K),\pa^-)$
is a transverse link invariant.

\begin{thm}\label{thm:filterinv}
The filtered DGA,
$(\KCA^-(K),\pa^-),$ is invariant under transverse isotopies of $K$ up to filtered stable tame isomorphism.
\end{thm}

\begin{pf}
By Lemma \ref{lma:shrink} we may assume that $K_t$, $0\le t\le 1$ is an
isotopy of transverse links inside $B(\delta)$
connecting two given transverse links.
Then $\Lambda_{K_t}$ is an isotopy of Legendrian submanifolds
confined to the region in $J^{1}(S)$ where $|p|\le\delta$.
To prove the invariance statement we generalize the invariance proof
in \cite{EkholmEtnyreSullivan05b}.

We study parameterized moduli spaces and first note that, as in
\cite[Lemma 2.11]{EkholmEtnyreSullivan05b}, when there are no disks of index $-1$ and no births/deaths
of intersection points, the moduli spaces change by cobordisms and the filtered differential is unchanged.

Suppose at some critical time $t'$ an index $-1$ disk exists. Like in \cite[Section 10]{EkholmEtnyreSullivan05b}, we use a small
perturbation of the trace of the fronts of the isotopies near the critical instance to
create a Legendrian submanifold
$\Lambda_{K_{t'}}\times\R\subset J^{1}(S\times\R)$
and study a compact part of this manifold corresponding to $[-1,1]\subset\R$.
Straightforward modifications  of Lemmas~\ref{lma:confine}, \ref{lma:subvariety},
and~\ref{lma:disjoint2} show that
$\overline{H}_{\pm}\times\C\subset T^{\ast}(S\times\R)$ (where we think of $T^{\ast}\R$ as $\C$)
are disjoint from
$\pi_\C(\Lambda_{K_{t'}}\times\R)$
and is complex in regions where holomorphic disks with one positive
puncture might exist.
As above it then follows that $\overline{H}_{\pm}\times\C$ give filtrations
on $\LCA^-(\Lambda_{K_{t'}}\times\R)$ compatible with those induced on $\LCA^{-}(\Lambda_{K_{t'}})$.
From this filtered differential we construct a filtered tame
isomorphism by repeating, essentially verbatim, the construction of
the tame isomorphism in the unfiltered case given in
\cite[Lemma 2.12]{EkholmEtnyreSullivan05b}.

When a birth/death of intersection points occurs, we use the same analytical ``degenerate" gluing results \cite[Proposition 2.16 and 2.17]{EkholmEtnyreSullivan05b} and a filtered version of the algebraic arguments in \cite[Lemma 2.15]{EkholmEtnyreSullivan05b} to construct an explicit filtered stable tame isomorphism.
\end{pf}

\subsection{Topological invariance of the infinity version in $\R^3$}
\label{subsec:infty}

Before delving in depth into the technical details of the computation of the filtered DGA in $\R^3$ in Section~\ref{sec:computations}, we give a geometric explanation
of Theorem~\ref{thm:topNOTtran}, which says that the
infinity version of the transverse invariant in $\R^3$ is actually a topological
link invariant. To simplify notation, we treat only the single-component knot case in this subsection. 
An alternative
discussion in the algebraic setting can be found in \cite{Ng10}, though the presentation here has the advantage that it explains the underlying geometric reason for this phenomenon.

Let $K$ be an oriented transverse knot in $\R^3$. The homology
$H_1(\Lambda_K) \cong \Z^2$ has a distinguished set of generators
corresponding to the meridian and ($0$-framed) longitude of $K$,
allowing us to identify $\Z[H_1(\Lambda_K)]$ with $R=\Z[\lambda^{\pm
  1},\mu^{\pm 1}]$. We
can then rewrite Equation~\eqref{FilteredDifferential.eq} as
\begin{equation} \label{FilteredDifferential.eq2}
        \partial^- a =
        \sum_{\{u \in (\mathcal{M}_A(a;{\mathbf b})/\R) \,\,|\,\,
        \dim\mathcal{M} = 1\}}
        (-1)^{|a|+1}
        \sigma(u)   U^{n_U(u)}V^{n_V(u)} \lambda^{\longitude(A)}\mu^{\mer(A)} {\mathbf b},
\end{equation}
where $A$ is the linear combination of $\longitude(A)$ longitudes and
$\mer(A)$ meridians.

As in \cite{Ng10}, define the infinity DGA,
$(\KCA^\infty(K),\pa^\infty),$ by tensoring $(\KCA^-(K),\pa^-)$
with $R[U^{\pm 1},V^{\pm 1}]$ and replacing $\lambda$ by $\lambda
(U/V)^{-(\selfl(K)+1)/2}$, where $\selfl(K)$ is the self-linking number of $K$.
We are now ready to prove Theorem~\ref{thm:topNOTtran} from the introduction.

\begin{proof}[Proof of Theorem~\ref{thm:topNOTtran}]
Let $u$ denote a holomorphic disk contributing to the Legendrian
contact homology of $\Lambda_K \subset \U^*\R^3$. Just as we viewed the boundary of $u$
in Subsection~\ref{subsec:conormal} as an element of $H_1(\Lambda_K)$
by appending capping paths at Reeb chords, we can view the entirety of
$u$ as an element $[u]$ of the relative homology $H_2(\U^*\R^3,\Lambda_K)$
by appending capping surfaces at Reeb chords.

Note that $\U^*\R^3 \cong J^1(S^2)$ is topologically $S^2 \times \R^3$. The exact sequence
\[
\cdots \longrightarrow 0 \longrightarrow H_2(\U^*\R^3)=\Z \longrightarrow
H_2(\U^*\R^3,\Lambda_K) \longrightarrow H_1(\Lambda_K) = \Z^2
\longrightarrow 0 \longrightarrow \cdots
\]
implies that $H_2(\U^*\R^3,\Lambda_K) \cong \Z^3$. Pick a basis
$s,l,m$ of $H_2(\U^*\R^3,\Lambda_K)$ such that the following holds:
\begin{itemize}
\item
$s = [S^2]$, the class of $S^2$ in $\U^*\R^3 \cong \R^3\times S^2$;
\item
$l$ is the homology class of the conormal lift of a cooriented Seifert surface of $K$;
\item
$m$ is the hemisphere of the $S^2$ fiber of $\U^*\R^3$ over some point
$p\in K$, bounded by the intersection of $\Lambda_K$ with this fiber
and containing the positive lift of $\xi$ over $p$.
\end{itemize}
Note that under the boundary map
$H_2(\U^*\R^3,\Lambda_K) \longrightarrow H_1(\Lambda_K)$, the classes
$s,l,m$ map to $0$, the longitude, and the meridian, respectively.
Now intersecting $H_\pm$ or projecting to $T^*S^2$ and intersecting with with $\overline{H}_{\pm}$
defines linear maps
\[
n_U,n_V : H_2(\U^*\R^3,\Lambda_K) \to \Z
\]
such that $n_U(s) = n_V(s) = 1$, $n_U(m)=1$, and $n_V(m)=0$. By the
definition of self-linking number, the difference
$n_U(l)-n_V(l)$ is $\selfl(K)$; by adding the appropriate multiple
of $s$ to $l$, we may assume that $n_U(l) = \frac{\selfl(K)-1}{2}$ and
$n_V(l) = -\frac{\selfl(K)+1}{2}$.

If we write $[u] = s(u)s+l(u) l+m(u)m$ for $s(u),l(u),m(u) \in\Z$,
then $n_U[u] = s(u)+\frac{\selfl(K)-1}{2}l(u)+m(u)$ and
$n_V[u] = s(u)-\frac{\selfl(K)+1}{2}l(u)$. From
Equation~\eqref{FilteredDifferential.eq2}, the contribution of $u$ to the
differential $\pa^\infty$ has coefficient
\begin{align*}
(-1)^{|a|+1} \sigma(u) U^{n_U[u]} V^{n_V[u]}&
(\lambda(U/V)^{-(\selfl(K)+1)/2})^{l(u)} \mu^{m(u)}\\ &=
(-1)^{|a|+1} \sigma(u) (\lambda/U)^{l(u)} (\mu U)^{m(u)} (UV)^{s(u)}.
\end{align*}

Now in the definition of the unfiltered Legendrian contact homology of
$\Lambda_K$, which is a
topological knot invariant, one could use the coefficient ring
$\Z[H_2(\U^*\R^3,\Lambda_K)]$ rather than $\Z[H_1(\Lambda_K)]$. If one
writes $\tilde{\lambda},\tilde{\mu},\tilde{\sigma}$ for the
multiplicative generators of
$\Z[H_2(\U^*\R^3,\Lambda_K)]$ corresponding to $l,m,s$, then the
contribution of $u$ to Legendrian contact homology with this enhanced
coefficient ring is
\[
(-1)^{|a|+1} \sigma(u) \tilde{\lambda}^{l(u)} \tilde{\mu}^{m(u)}
\tilde{\sigma}^{s(u)}.
\]
But this is precisely the coefficient of the contribution of $u$ to
$\pa^{\infty}$, once we make the global substitutions $\tilde{\lambda}
= \lambda/U$, $\tilde{\mu} = \mu U$, $\tilde{\sigma} = UV$.

It follows that this global substitution turns
$(\KCA^\infty(K),\pa^\infty)$ into the unfiltered Legendrian contact homology DGA
of $\Lambda_K$ with coefficients in
$\Z[H_2(\U^*\R^3,\Lambda_K)]$. Since the latter is a topological
invariant, the result follows.
\end{proof}

We remark that our choice of basis $(l,m,s)$ for
$H_2(S^*\R^3,\Lambda_K)$ in the proof of Theorem~\ref{thm:topNOTtran}
is canonical, depending only on the topological type of $K$;
as a consequence, the (stable tame) isomorphisms on
Theorem~\ref{thm:topNOTtran} act as the identity, not just an
isomorphism, on the base ring
$\Z[\lambda^{\pm 1},\mu^{\pm 1},U^{\pm 1},V^{\pm 1}]$. This is clear
for $m$, which can be defined using the orientation on $K$ rather than
the contact structure, and for $s$. For $l$, suppose that there is a
topological
isotopy $K_t$ between two transverse knots $K_0$ and $K_1$, with a
corresponding isotopy $\Sigma_t$ of Seifert surfaces, such that $K_t$
fails to be a transverse knot at finitely many moments. At these moments,
the self-linking number may jump, generically by $\pm 2$, but $n_U([\Sigma_t])$
and $n_V([\Sigma_t])$ also each jump by $\pm 1$. Thus the quantity
$n_U([\Sigma_t]) - \frac{\selfl(K_t)-1}{2}$
remains unchanged during
the isotopy. Since $l$ (which is $[\Sigma_t]$ plus some multiple
of $s$) is chosen so that $n_U(l) = \frac{\selfl(K_t)-1}{2}$, the isotopy $K_t$
preserves $l$.

\section{Computing the Filtered DGA of a Transverse Link} \label{sec:computations}

In this section we compute the filtered DGA of a transverse link to prove Theorem \ref{thm:mainlinktv}. We begin by describing the main strategy used in \cite{EENS} for calculating knot contact homology. This description leads to a sufficient understanding of the behavior of all holomorphic disks needed for the calculation of the filtered differential.

\subsection{Scheme for calculating knot contact homology}\label{scheme}
The calculation of knot contact homology in \cite{EENS} proceeds as follows. Consider a link $K$ braided around the unknot $U$ with $n$ strands and view it as a multisection of a fibration $U\times D^{2}$, where $D^{2}$ is the $2$-disk, corresponding to a tubular neighborhood of $U$. As we degenerate the multisection toward $U$ the conormal lift $\Lambda_K$ approaches  the conormal lift of the unknot $\Lambda_U$ (with multiplicity $n$). More precisely, a neighborhood of $\Lambda_U\subset S^{\ast}\R^{3}$ is contactomorphic to the $1$-jet space $J^{1}(\Lambda_U)$ of $\Lambda_U$ and for $K$ sufficiently close to $U$, $\Lambda_K$ is a multisection of $J^{1}(\Lambda_U)\to\Lambda_U$ with $n$ sheets. In other words, over any open disk $W\subset \Lambda_U$, $\Lambda_K$ is given by the $1$-jet extension of $n$ functions $F_j\colon W\to\R$, $j=1,\dots,n$.

It will be useful to recall that when a Legendrian submanifold $\Lambda$ in $J^1(S^2)$ is given locally as the $1$-jet of $n$ functions $F_i$ then its Reeb chords (that is double points of $\pi_\C(\Lambda)$) correspond to critical points with positive critical values of the difference of these local functions $F_i-F_j.$

By  \cite[\S 3]{EENS}, 
we know that close to the limit as $\Lambda_K$ approaches $\Lambda_U,$ the Reeb chords of $\Lambda_K$ are of two types:
\begin{itemize}
\item[$\mathbf{I}.$] Near each Reeb chord $ch$ of $\Lambda_U$ there are $n^{2}$ Reeb chords $ch_{ij}$, $1\le i,j\le n$ of $\Lambda_K$, where we write $ch_{ij}$ for the chord near $ch$ that starts on the $j^{\rm th}$ local sheet of $\Lambda_K$ near the start point of $ch$ and ends at the $i^{\rm th}$ local sheet near the endpoint of $ch$.
\item[$\mathbf{II}.$] There are $n(n-1)$ small Reeb chords corresponding to critical points of positive local function differences of the form $F_i-F_j$. These critical points are either maxima or saddle points. We denote the former by $b_{ij}$ and the latter by $a_{ij}$, where $1\le i,j\le n$, $i\ne j$, and we use the same notational conventions for subscripts as above.
\end{itemize}
Thus Reeb chords of $\Lambda_K$ are either Reeb chords of $\Lambda_U$ with a small chord added or subtracted or small Reeb chords entirely inside the neighborhood of $\Lambda_U.$

One of the main technical results of \cite{EENS} (see \cite{EESa} for a similar result) shows that holomorphic disks admit a similar description. Near the limit as $\Lambda_K$ approaches $\Lambda_U$, rigid holomorphic disks in $T^{\ast} S$ with boundary on $\overline{\Lambda}_K = \pi_\C(\Lambda_K)$ and one positive puncture are of two types:
\begin{itemize}
\item[$\mathbf{I}.$] They can lie in an arbitrarily small neighborhood of the union of the following: a disk with one positive puncture and boundary on $\Lambda_U;$ and, certain flow trees of the function differences $F_i-F_j$ attached along the disk's boundary.
\item[$\mathbf{II}.$] They can lie entirely inside a small neighborhood of $\Lambda_K$ and are given locally as flow trees of the functional differences $F_i-F_j$.
\end{itemize}
Furthermore, any disk with its positive puncture at a chord of type
$\mathbf{I}$ (resp.~$\mathbf{II}$) is of type $\mathbf{I}$
(resp.~$\mathbf{II}$). For the notion of flow trees we refer to
\cite[Section 2]{Ekholm07}. We do not give a complete definition of
flow trees here, but merely note that they are made from pieces of
flow lines of the functional differences $F_i-F_j$ and that as
$\Lambda_K$ collapses onto $\Lambda_U$ the corresponding holomorphic
disks stay in smaller and smaller neighborhoods of $\Lambda_U.$

In conclusion, in order to compute the differential $\pa\colon\KCA(K)\to\KCA(K)$ we need to understand holomorphic disks with boundary on $\overline{\Lambda}_U$ and flow trees determined by $F_i-F_j$. The latter can be understood using finite dimensional Morse theory. In order to understand the former we use the correspondence between holomorphic disks and flow trees on fronts from \cite{Ekholm07}.

\subsection{The conormal lift of the unknot}\label{ssec:unknot}
One may easily compute (or see \cite{EkholmEtnyre05,EENS} and
Lemma~\ref{liftU} below) the conormal lift $\Lambda_U$ of the unknot
$U$ represented as the circle of radius $r<\delta$ in the
$x_1x_2$-plane. This lift can be slightly perturbed in $J^1(S)$ so
that it has two Reeb chords denoted $c$ and $e,$ see Figure
\ref{unknotlift}.
\begin{figure}[htb]
  \relabelbox \small {
  \centerline{\epsfbox{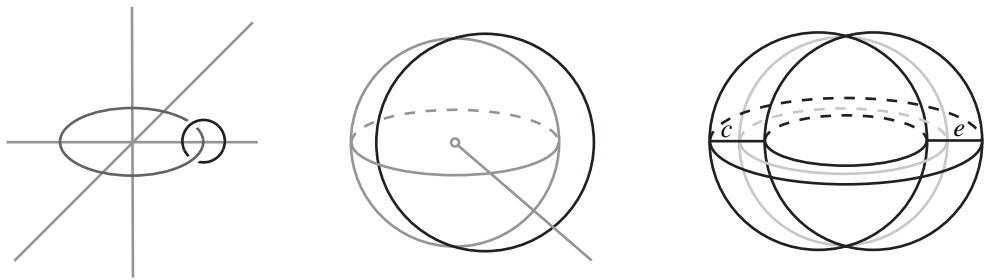}}}
  \relabel{e}{$e$}
  \relabel{c}{$c$}
  \endrelabelbox
\caption{On the left is $\R^3$ with the unknot in the $xy$-plane and a normal
  circle at one point is shown. The middle figure shows the normal
  circle seen in $S^2\times\R\cong \R^3-\{(0,0,0)\}$ which is thought
  of as the (front) projection of $S^*\R^3\cong J^1(S^2)$ to
  $S^2\times \R.$ On the right is the front projection of the entire
  conormal lift of the unknot which is obtained by rotating the circle
  shown in the middle figure about the line through the north and
  south poles. We have also slightly perturbed the picture on the
  right so that there are only two Reeb chords, labeled $c$ and $e$ in
  the figure.
}
\label{unknotlift}
\end{figure}

For purposes of finding holomorphic disks via flow trees, $\Lambda_U$
must be perturbed to be in general position with respect to the front
projection into $S^2\times \R.$ Notice that in Figure~\ref{unknotlift}
there is a circle in $\Lambda_U$ that is mapped to the north pole and
a circle mapped to the south pole. Since $\Lambda_U$ is already in
general position outside the fibers over the north and south pole of
$S$ we concentrate our attention there. Near the poles, the projection
$\overline{\Lambda}_U$ looks as described in the following lemma.
\begin{lma}\label{liftU}
If $U$ is the circle of radius $r<\delta$ in the $x_1x_2$-plane, then
\[
\overline{\Lambda}_U\cap T^{\ast}E_{2\delta_0}=\{(q,p)\colon q=(\xi_1,\xi_2,\pm 1),\,p=(x_1,x_2,0)\}\cong S^{1}\times(-1,1) \times \{\pm 1\},
\]
where $\sqrt{x_1^{2}+x_2^{2}}=r$, $-\xi_1x_2+\xi_2x_1=0$, and $\cong$ denotes diffeomorphism.
\end{lma}
\begin{pf}
This is immediate from the definition.
\end{pf}
\begin{figure}[htb]
  \relabelbox \small {
  \centerline{\epsfbox{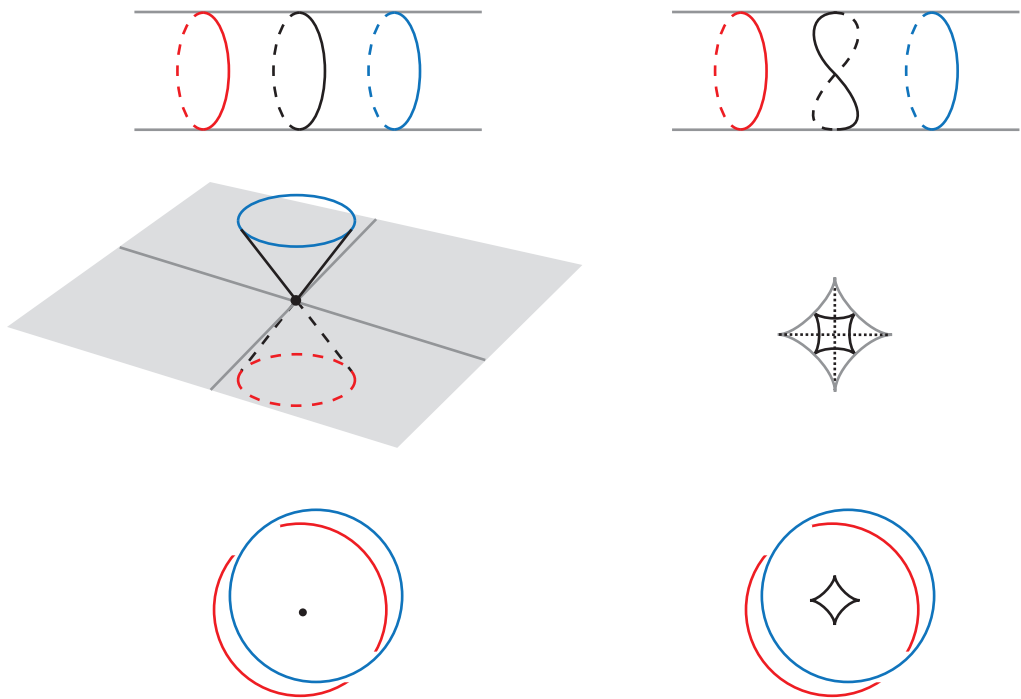}}}
  \endrelabelbox
\caption{Along the top of the figure is an annular neighborhood of the circle in $\Lambda_U$ that maps to the fiber above the north pole.
On the middle left, we see the image of this annulus near
the north pole in the front projection,
a cone whose boundary is two circles.
On the bottom left is the image of this annulus near the north pole in $S^2$ (that is, the top view of the cone where we have slightly offset the circles so that they are both visible).
On the middle right, we see the top view of the cone after it has been perturbed to have a generic front projection. More specifically, the lighter outer curve is the image of the cusp curves, the dotted lines are the image of double points in the front projection and the darkest inner curve is the image of the circle that mapped to the cone point before the perturbation.
On the bottom right, we see the image in $S^2$ of the cusp curve and
the two boundary circles on $\Lambda_U$.}
\label{fig:resolution}
\end{figure}
In Figure \ref{fig:resolution} we see the front projection of
$\Lambda_U$ over the region where $S$ is flat.
The left pictures show $\Lambda_U$ in a neighborhood of the circle
over the north (or south) pole, as described in Lemma~\ref{liftU}, and
the corresponding cone in the front projection of $\Lambda_U$. The
right pictures show $\Lambda_U$ after small perturbation near the cone
point that makes the front projection generic.
Using the right representation we get the following result describing holomorphic disks of $\Lambda_U$, taken from \cite[\S 3]{EENS}. See Figure~\ref{fig:diskboundaries} for a description of the flow trees on the front
(which are close to the projection of the disk boundaries) and Figure \ref{fig:rigidtrees} for a description of their lifts into $\Lambda_U$.
\begin{figure}[htb]
  \relabelbox \small {
  \centerline{\epsfbox{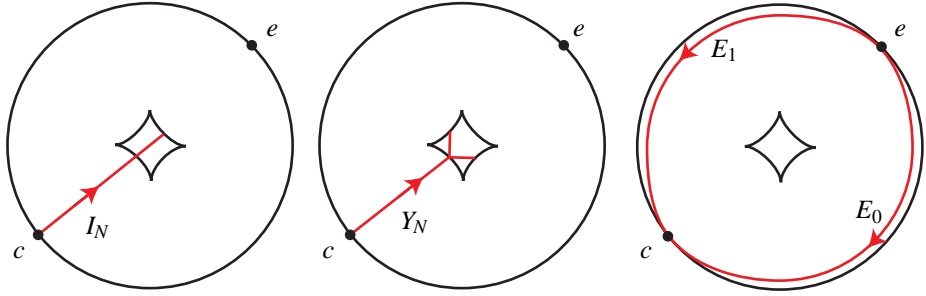}}}
  \relabel{1}{$c$}
  \relabel{2}{$c$}
  \relabel{3}{$c$}
  \relabel{4}{$e$}
  \relabel{5}{$e$}
  \relabel{6}{$e$}
  \relabel{I}{$I_N$}
  \relabel{Y}{$Y_N$}
  \relabel{e0}{$E_0$}
  \relabel{e1}{$E_1$}
  \endrelabelbox
\caption{The rigid disks from Lemma~\ref{lma:UnknotDGA} in the northern hemisphere of $S^2.$}
\label{fig:rigidtrees}
\end{figure}

\begin{figure}[htb]
  \relabelbox \small {
  \centerline{\epsfbox{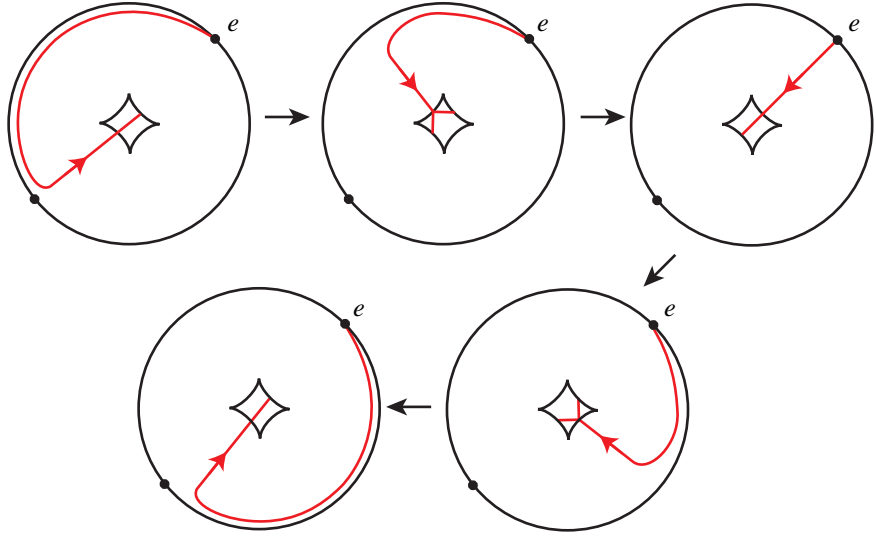}}}
  \relabel{1}{$e$}
  \relabel{2}{$e$}
  \relabel{3}{$e$}
  \relabel{4}{$e$}
  \relabel{5}{$e$}
  \endrelabelbox
\caption{The one dimensional families of trees from Lemma~\ref{lma:UnknotDGA} in the northern hemisphere of $S^2.$}
\label{fig:onedimtreeunknot}
\end{figure}

\begin{lma}[Ekholm--Etnyre--Ng--Sullivan 2011, \cite{EENS}]\label{lma:UnknotDGA}
There are exactly six rigid holomorphic disks with boundary on
$\overline\Lambda_U$: four ($I_N$,
$Y_N$, $I_S$, and $Y_S$) with positive puncture at $c$ and no negative puncture, and two
($E_1$,$E_2$) with positive puncture at $e$ and
negative puncture at $c$. If $q$ is any point in $\Lambda_{U}$ lying over a point where the front of $\Lambda_U$ has 2 sheets then there are exactly two constrained rigid holomorphic disks with positive puncture at $e$ and boundary constrained to pass through $q$. These two disks correspond to two constrained rigid flows which lie in two of the $1$-parameter families $\widetilde I_N$, $\widetilde Y_N$, $\widetilde I_S$, and
$\widetilde Y_S$ of flow trees with positive puncture at $e$, and are rigidified by the condition that their $1$-jet lift passes through $q$.
The boundaries of these $1$-parameter families are
as follows:
\begin{align*}
\pa \widetilde I_N = (E_1\,\#\, I_N) \cup (E_2\,\#\, I_N),&\quad
\pa \widetilde Y_N = (E_1\,\#\, Y_N) \cup (E_2\,\#\, Y_N),\\
\pa \widetilde I_S = (E_1\,\#\, I_S) \cup (E_2\,\#\, I_S),&\quad
\pa \widetilde Y_S = (E_1\,\#\, Y_S) \cup (E_2\,\#\, Y_S).
\end{align*}
Here $E_1\,\#\, I_N$ denotes the broken flow tree obtained by adjoining $I_N$ to $E_1$ etc., see Figure \ref{fig:onedimtreeunknot}.

Using the capping path convention and isomorphism
$H_1(\Lambda_U) = \Z \langle \mu, \lambda \rangle$ specified in \cite{EENS}, the disks $I_N, Y_N, I_S,Y_S$
contribute   $1, \lambda, \lambda\mu, \mu,$ respectively, to $\pa c.$
Similarly $E_1$ and $E_2$ contribute $-\lambda^{-1}c$ and
$\lambda^{-1}c,$ respectively, to $\partial e.$ This gives the
differential
\begin{align*}
\pa c & =  1 + \lambda + \lambda \mu + \mu \\
\notag
\pa e & =  (-\lambda^{-1} + \lambda^{-1}) c = 0.
\end{align*} 
\end{lma}

As in \cite[\S 3]{EENS}, it suffices to consider constrained rigid trees when studying the filtered contributions of the curves. 
(See the third result of Theorem~\ref{thm:close} below.)
Although not needed for the purposes of calculating the filtered differential in this paper, it is possible, as mentioned in 
\cite[\S 3]{EENS},  to prove that 1-parameter families of holomorphic disks with positive puncture at $e$ are in natural one-to-one correspondence with the $1$-parameter families of flow trees mentioned in Lemma \ref{lma:UnknotDGA}.

\subsection{Intersection numbers for disks with boundary on $\overline{\Lambda}_U$}\label{filtereduk}

Now that we have recalled the computation of the knot contact homology of the standard unknot in the $x_1x_2$-plane, we turn to computing the filtration on this knot thought of as a transverse knot.

\begin{thm}
Let $U$ be the transverse unknot with self-linking number $-1.$ The
filtered DGA $(\KCA^-(U),\pa^-)$ is filtered stable tame isomorphic to
the algebra over $R[U,V]$ generated by $c$ and $e,$ where $|c|=1$ and
$|e|=2,$ and
\label{thm:unknot}
\[
\pa^- c= 
U+\lambda+\lambda\mu V+\mu \quad \text{ and }
\quad \partial^- e=0.
\]
\end{thm}

\begin{proof}

From the explicit description of the knot contact homology differential in Lemma~\ref{lma:UnknotDGA}, it suffices to compute the intersection numbers of the four holomorphic disks $I_N,Y_N,I_S,Y_S$ with $\overline{H}_\pm.$
To this end, we will compute the intersection numbers of $\overline{H}_\pm$ with a disk $D_+$ homotopic to ``$Y_N - I_N$" to determine the  relative intersection numbers of $Y_N$ and $I_N$ with $\overline{H}_\pm.$ The absolute intersection numbers will follow from positivity of intersections. A similar argument will apply to $I_S$ and $Y_S.$

We first note that Lemma~\ref{lma:disjoint2} implies that any point in
the intersection between $\overline{H}_\pm$ and a holomorphic disk
with boundary on $\Lambda_U$ must lie in $T^{\ast}E_{2\delta_0}$ so it
will be sufficient to consider the parts of the holomorphic disks
which lie in this region.  As we shrink the unknot $U$ towards the
$x_3$ axis in $\R^3$, its conormal lift $\Lambda_U$ approaches the
$0$-section in $J^{1}(S).$ We can also see that the perturbed front
generic version of $\Lambda_U$ collapses onto the $0$-section. As we
degenerate $\Lambda_U$ onto the $0$-section, the boundaries of the
holomorphic disks converge to the curves on the torus $\Lambda_U$ depicted in Figure~\ref{fig:diskboundaries}.
\begin{figure}[htb]
  \relabelbox \small {
  \centerline{\epsfbox{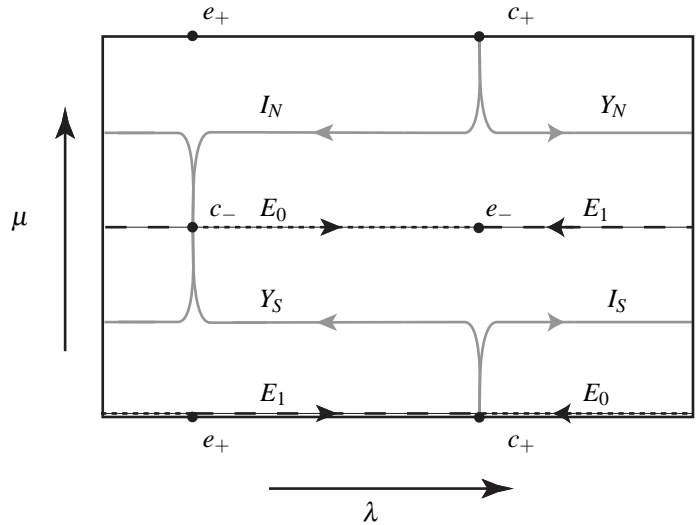}}}
  \relabel{c1}{$c_+$}
  \relabel{c}{$c_-$}
  \relabel{c2}{$c_+$}
  \relabel{e1}{$e_+$}
  \relabel{e2}{$e_+$}
  \relabel{e}{$e_-$}
  \relabel{i1}{$I_N$}
  \relabel{i2}{$I_S$}
  \relabel{y1}{$Y_N$}
  \relabel{y2}{$Y_S$}
    \relabel{e11}{$E_0$}
        \relabel{e12}{$E_0$}
  \relabel{e22}{$E_1$}
    \relabel{e21}{$E_1$}
  \relabel{m}{$\mu$}
  \relabel{l}{$\lambda$}
  \endrelabelbox
\caption{Boundaries of holomorphic disks on $\Lambda_U$.
}
\label{fig:diskboundaries}
\end{figure}

The Legendrian torus $\Lambda_U$ can be described as the $1$-jet of a multifunction from $S^2$ to $\R$ and we notice that
the curves in Figure~\ref{fig:diskboundaries} are given by gradient flow lines for the functional differences of the multifunction.
Furthermore, for an appropriate almost complex structure, the holomorphic disks $C^{1}$-converge to the strips corresponding to these flow lines outside any fixed neighborhood of its vertices, see \cite[Lemma 5.13, Remark 5.14, and Subsection 6.4]{Ekholm07}, and inside neighborhoods of its vertices they converge to other local models, see \cite[Subsection 6.1]{Ekholm07}. Here the strip of a flow line consists of the line segments in the cotangent fibers between its cotangent lift. In particular, if we choose a perturbation of $\Lambda_U$ so that its projection to $T^{\ast} S$ consists of affine subspaces in a neighborhood of $(0,0,1)\in S$ then the standard complex structure for which $\overline{H}_{\pm}$ is a complex hypersurface is appropriate in the above sense, see \cite[Subsection 4.1, $2^{\rm nd}$ bullet from the end]{Ekholm07}.

Consider the
flow-line disks depicted in Figure \ref{fig:intersectionnumbers}.
These flow
lines lie in the two distinct homotopy classes of the disks $I_N$ and
$Y_N$. (Which of these disks looks like the letter $I$
or $Y$ depends on the perturbation of $\Lambda_U$ which makes
its front generic, see also Figure~\ref{fig:onedimtreeunknot}.)
\begin{figure}[htb]
  \relabelbox \small {
  \centerline{\epsfbox{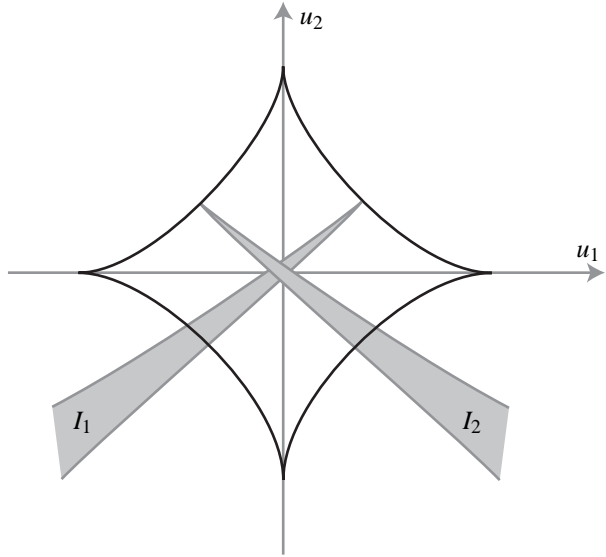}}}
  \relabel{1}{$I_1$}
  \relabel{2}{$I_2$}
  \relabel{u1}{$u_1$}
  \relabel{u2}{$u_2$}
  \endrelabelbox
\caption{Two holomorphic disks in distinct homotopy classes. We have chosen coordinates on $S^2$ near the north pole so that the north pole corresponds to $(0,0).$}
\label{fig:intersectionnumbers}
\end{figure}
The disks $I_1$ and $I_2$ correspond to flow lines of the difference
of two functions that locally describe part of $\Lambda_U$ near the
north pole.
With notation and coordinates near the north pole of $S^2$ indicated in Figure~\ref{fig:intersectionnumbers}, the gradient of the function describing the upper sheet of $I_1$ equals $-\pa_{u_1}$ and that of its lower sheet is $\pa_{u_2},$ and the flow line corresponding to  $I_1$ in $S$ follows the curve $C_1=(t,t)$, $-T\le t\le \epsilon$ where $\epsilon>0$ is arbitrarily small and not yet fixed. Thus the flow line strip of $I_1$ is
\[
(t,s)\mapsto ((t,t), (s(-\pa_{u_1})+(1-s)\pa_{u_2})), \quad -T\le t\le \epsilon,\,\,0\le s\le 1,
\]
near the origin. Similarly, the flow line strip of $I_2$ is
\[
(t,s)\mapsto ((-t,t), (s\pa_{u_1}+(1-s)\pa_{u_2})), \quad -T\le t\le \epsilon,\,\,0\le s\le 1.
\]
On the other hand the intersection of $\overline{H}_+$ with the fibers of $T^{\ast} S$ over the point $(t,t)\in C_1$ is $a( t\pa_{u_1}-t\pa_{u_2})$ for some $a>0$, see Equation~\eqref{eq:localH}. That is,
\[
\overline{H}_+\cap T^{\ast}C_1=\{((t,t), a(t\pa_{u_1}-t\pa_{u_2}))\},
\]
and we see that the strip and the hypersurface intersect once (over the point $(-\frac{1}{2a},-\frac{1}{2a})$). Similarly, over the flow line of $I_2$ we have
\[
\overline{H}_+\cap T^{\ast}C_1=\{((-t,t), a(t\pa_{u_1}+t\pa_{u_2}))\},
\]
and we see that the strip and the hypersurface do not intersect: the solution of the equation which sets the fiber coordinates equal would lie over the point $(-t,t)$ where $t=\frac{1}{2a}$ but if we
set $\epsilon < \frac{1}{2a},$ this point does not lie on the flow line of $I_2$.

In order to relate the above calculation to Lemma \ref{lma:UnknotDGA},	 consider the disk $D_\pm$ in the fiber of $T^*S$ bounded by the circle $((0,0,\pm1),p)\in \overline\Lambda_U\cap T^{\ast}E_{2\delta_0}$
and oriented according to the induced orientation on the fiber. We notice that this circle is a lift of a longitude for $U.$
Using the description of $\overline H_{\pm}$ in Lemma~\ref{lma:subvariety} it is easy to see that $D_\pm$ intersects $\overline H_{\pm}$ with intersection number $\mp 1.$  (Note here that the orientation of the base followed by the
orientation of the fiber gives the orientation opposite to the complex
orientation on $T^{\ast}S$ and that the orientation induced on the
normal bundle to the fiber by $\overline H_{+},$ respectively
$\overline H_{-},$ agrees, respectively disagrees, with the
orientation on the base.)
Consulting Figure~\ref{fig:diskboundaries} and considering the
algebraic topology of $\Lambda_U\subset T^*S^2$ one may easily see
that the difference cycle between the disk corresponding to $I_N$ and
$Y_N$ is homologous to $D_+$. Since the intersection number of these
disks with $\overline{H}_+$ equals $1$ or $0$ (by the above
calculation)  it follows that the disk corresponding to $I_N$ has
intersection number $1$ with $\overline{H}_+$ and the disk
corresponding to $Y_N$ has intersection number $0$. Similarly, the
disks corresponding to $I_S$ and $Y_S$ intersect $\overline{H}_-$ with
intersection number $1$ and $0$, respectively. In addition, since
$D_{\mp}$ does not intersect $\overline{H}_{\pm}$, we find that the
$I_S$ and $Y_S$ (respectively $I_N$ and $Y_N$) disks do not intersect
$\overline{H}_+$ (respectively $\overline{H}_-$).

\end{proof}


\subsection{The filtered DGA of a transverse link}\label{ssec:proof}
We begin by recalling the computation of the knot contact homology from \cite[Theorem 1.1]{EENS}. Using the notation in the introduction we have the following result.
\begin{thm}[Ekholm--Etnyre--Ng--Sullivan 2011, \cite{EENS}] \label{KnotDGA.thm}
The differential in the Legendrian DGA associated to the conormal lift $\Lambda_K$ of a framed knot $K$ is $(\KCA(K),\pa)$ where $\KCA(K)$ is generated over $\Z[\mu^{\pm}, \lambda^{\pm}]$ by the $a_{ij},b_{ij}, c_{ij},$ and $e_{ij}$ described in  Subsection~\ref{scheme} and~\ref{ssec:unknot} and
the map $\pa\colon \KCA(K)\to\KCA(K)$ is determined by the following matrix equations:
\begin{align*}
\partial \mathbf{A} &= 0, \\
\partial \mathbf{B} &= -\Ll^{-1}\cdot\Aa\cdot\Ll\,\, +\,\, \Phi^L_B \cdot \Aa \cdot \Phi^R_B, \\
\pa \mathbf{C} &= \Aa\cdot\Ll\,\, +\,\, \Aa\cdot\Phi^{R}_B, \\
\pa \mathbf{E} &= \Bb\cdot(\Phi_B^{R})^{-1}\,\, +\,\, \Bb\cdot\Ll^{-1}\,\, -\,\,
\Phi^L_B\cdot\Cc\cdot\Ll^{-1}\,\, +\,\, \Ll^{-1}\cdot\Cc \cdot (\Phi^R_B)^{-1},
\end{align*}
where $\mathbf{A},\mathbf{B},\mathbf{C},\mathbf{E},\Ll,\Phi^L_B,\Phi^R_B$ are as in
Subsection~\ref{subsec:computation}, and if $\Mm$ is an $(n\times
n)$-matrix, the matrix $\pa \Mm$ is defined by $(\pa \Mm)_{ij}=\pa
\Mm_{ij}$.
\end{thm}
In Subsection~\ref{scheme} above we discussed the holomorphic curves involved in the computation of the differential. In particular, one may easily conclude the following result.
\begin{lma}[Ekholm--Etnyre--Ng--Sullivan 2011, \cite{EENS}]\label{lem:close}
Given any $\delta'>0,$ $\Lambda_K$ can be Legendrian isotoped to be close enough to $\Lambda_U$ so that any
holomorphic disk with boundary on $\overline\Lambda_K,$
one positive puncture, and involving only the chords $a_{ij}$ and $b_{ij},$ has its image contained within a $\delta'$-neighborhood of $\overline\Lambda_U.$
\end{lma}
\begin{proof}
While this lemma follows from the results in  \cite[\S 3.4]{EENS}, we comment that it can also be seen by observing that the Reeb chords $a_{ij}$ and $b_{ij}$ have small action. This action can be made arbitrarily small as we isotop $K$ to be close to $U.$ Now a monotonicity argument will confine the holomorphic curves to stay close to $\overline\Lambda_U.$
\end{proof}
To compute the filtration on a transverse knot we will need to explicitly describe the holomorphic disks used in the computation of the differential. More specifically we need to understand disks of Type {\rm I}, discussed in Subsection~\ref{scheme}. To this end we summarize the computations from \cite[\S 4.4]{EENS}.

\begin{thm}[Ekholm--Etnyre--Ng--Sullivan 2011, \cite{EENS}] \label{thm:close}
Given any $\delta'>0,$ $\Lambda_K$ can be Legendrian isotoped to be close enough to $\Lambda_U$ so that any
rigid holomorphic disk with boundary on $\overline\Lambda_K$ either has its image contained within a $\delta'$-neighborhood of $\overline\Lambda_U$ and one of the rigid disks described in Lemma~\ref{lma:UnknotDGA}, or has its image contained within a $\delta'$-neighborhood of $\overline\Lambda_U$ and one of the constrained rigid disks described in Lemma~\ref{lma:UnknotDGA} constrained by the projection $q \in \Lambda_{U}$ of one of the endpoints of a Reeb chord $b_{ij}$.
\begin{enumerate}
\item
The holomorphic disks that contribute to the terms in $\partial c_{ij}$ satisfy
\begin{itemize}
\item
the terms in $\Aa \Ll$ with $\mu$ coefficients are contained in neighborhoods of  $I_S$ (and $\overline\Lambda_U$),
\item
the terms in  $\Aa \Ll$ without $\mu$ coefficients
 are contained in neighborhoods of   $Y_N$ (and $\overline\Lambda_U$),
\item
the terms in  $\Aa \Phi^R_B$ with $\mu$ coefficients
 are contained in neighborhoods of  $Y_S$ (and $\overline\Lambda_U$), and
\item
the terms in  $\Aa \Phi^R_B$ without $\mu$ coefficients
 are contained in neighborhoods of  $I_N$ (and $\overline\Lambda_U$).
\end{itemize}

\item
The holomorphic disks that contribute to the $\Phi^L_B\cdot\Cc\cdot\Ll^{-1}\,\, +\,\, \Ll^{-1}\cdot\Cc \cdot (\Phi^R_B)^{-1}$  terms in $\partial e_{ij}$  are contained in neighborhoods of  either  $E_0$ or $E_1$ (and $\overline\Lambda_U$).

\item
The holomorphic disks that contribute to the $\Bb\cdot(\Phi_B^{R})^{-1}\,\, +\,\, \Bb\cdot\Ll^{-1}$ terms of $\pa e_{ij}$ are  contained in neighborhoods of one of the disks in the 1-parameter families $\widetilde I_N$, $\widetilde Y_N$, $\widetilde I_S$, and $\widetilde Y_S$ (and $\overline\Lambda_U$). That is, they are close to the union of $\overline\Lambda_U\cup E_1\cup E_2$ and one of the disks $I_N, Y_N, I_S, Y_S,$ shifted in the $\lambda$-direction some fixed distance (the same distance for all disks). More precisely,
\begin{itemize}
\item
disks associated to those terms in $\Bb \Ll^{-1}$ with $\mu$ coefficients
correspond to $Y_S,$
\item
disks associated to those terms in $\Bb \Ll^{-1}$ without $\mu$ coefficients
correspond to $I_N,$
\item
disks associated to those terms in $\Bb \left(\Phi^R_B\right)^{-1}$ with $\mu$ coefficients
correspond to $I_S,$ and
\item
disks associated to those terms in $\Bb \left(\Phi^R_B\right)^{-1}$ without $\mu$ coefficients
correspond to $Y_N.$
\end{itemize}
\end{enumerate}
\end{thm}

We are now in a position to prove the combinatorial expression for the
filtrations given in Theorem~\ref{thm:mainlinktv}.

\begin{proof}[Proof of Theorem~\ref{thm:mainlinktv}]
The first equation in the theorem follows for grading reasons.

The second equation in the theorem follows from Theorem~\ref{KnotDGA.thm} and the fact that none of the holomorphic disk contributing to the differential of a $b_{ij}$ can intersect $\overline{H}_\pm$ by Lemma~\ref{lem:close}.

Since the disks $E_0$ and $E_1$ lie close to the equator, Lemma~\ref{lma:disjoint2} and the third paragraph of Theorem~\ref{thm:close} imply
that  there are no intersections of $\overline{H}_{\pm}$
with the disks with positive puncture at an $e_{ij}$-chord and one negative puncture at an $c_{ij}$-chord.
Thus,
$\Phi^L_B\cdot\Cc\cdot\Ll^{-1}\,\, +\,\, \Ll^{-1}\cdot\Cc \cdot (\Phi^R_B)^{-1}$ appears in both $\pa \mathbf{E}$ and
${\pa}^- \mathbf{E}.$

The remainder of the differential involves disks that intersect $\overline{H}_{\pm}$. The formula for the filtered differential that derives from these intersections is easily derived by the intersections of
$\overline{H}_\pm$ with the disks from Lemma~\ref{lma:UnknotDGA} which
were worked out in Subsection~\ref{filtereduk}. Some simple bookkeeping
yields the desired computation.
\end{proof}

\section{Some Examples}\label{sec:ex}

In this section, we present some computations of the filtered
transverse knot DGA, $(\KCA^-,\pa^-)$. The first shows that the additional filtration structure is nontrivial; the second is an outline of a computation that shows that the filtered DGA is an effective invariant of transverse knots.

\subsection{The unknot}
We compute the transverse knot DGA for three versions of the unknot: the closure of the trivial $1$-braid and the closure of the $2$-braids $\sigma_1$ and $\sigma_1^{-1}$. The first two both represent the standard transverse unknot in $\R^3$ with self-linking number $-1$, while the third represents the transverse unknot with self-linking number $-3$. We will show that the filtered DGAs for the first two are stable tame isomorphic and are distinct from the filtered DGA for the third.

For the trivial $1$-braid, Theorem~\ref{thm:mainlinktv} or \ref{thm:unknot} yields a DGA with two generators $c,e$ and differential
\[
\pa^- c = U+\lambda+\lambda\mu V+\mu,
\quad
\pa^- e = 0.
\]

For the $2$-braid $\sigma_1$, the relevant matrices are
\[
\Phi^L_{\sigma_1} = \begin{pmatrix} -a_{21} & 1 \\
            	1 & 0
           \end{pmatrix}
,\quad
\Phi^R_{\sigma_1} = \begin{pmatrix} -a_{12} & 1 \\
            	1 & 0
           \end{pmatrix}
,\quad
\Ll = \begin{pmatrix} \lambda \mu & 0 \\
	0 & 1
      \end{pmatrix}
\]
\[
 \Aa = \begin{pmatrix}
	1+\mu &  a_{12} \\
	\mu a_{21} & 1+\mu
             \end{pmatrix}
, \quad
 \Aa^U = \begin{pmatrix}
	U+\mu &  U a_{12} \\
	\mu a_{21} & U +\mu
             \end{pmatrix}
, \quad
 \Aa^V = \begin{pmatrix}
	1+\mu V & a_{12} \\
	 \mu V a_{21} & 1+\mu V
             \end{pmatrix},
\]
\[
 \Bb^U = \begin{pmatrix}
	0 &  U b_{12} \\
	 \mu b_{21} & 0
             \end{pmatrix}
, \quad
 \Bb^V = \begin{pmatrix}
	0 & b_{12} \\
	\mu V b_{21} & 0
             \end{pmatrix}.
\]
We then calculate from Theorem~\ref{thm:mainlinktv} that the DGA for $\sigma_1$ has generators $a_{12}$,
$a_{21}$, $b_{12}$, $b_{21}$, $c_{11}$, $c_{12}$, $c_{21}$, $c_{22}$, $e_{11}$, $e_{12}$, $e_{21}$, $e_{22}$, and differential given by
$\pa^-(a_{12})=\pa^-(a_{21})=0$,
\begin{align*}
\pa^- \begin{pmatrix}
	0 & b_{12} \\
	 b_{21} & 0
             \end{pmatrix} &=
\begin{pmatrix}
	0 & -\frac{1}{\lambda\mu} a_{12}-a_{21} \\
	-\lambda\mu a_{21}-a_{12} & 0
             \end{pmatrix},
\\
\pa^- \begin{pmatrix}
c_{11} & c_{12} \\
c_{21} & c_{22}
\end{pmatrix} &=
\begin{pmatrix}
\lambda\mu+\lambda\mu^2 V-\mu a_{12} & \mu+U+a_{12} \\
\mu+U+\lambda\mu^2 V a_{21}-\mu a_{21}a_{12} & 1+\mu V+\mu a_{21}
\end{pmatrix},
\\
\pa^- \begin{pmatrix}
e_{11} & e_{12} \\
e_{21} & e_{22}
\end{pmatrix} &=
\begin{pmatrix}
b_{12}+\frac{1}{\lambda\mu}(c_{12}-c_{21}+a_{21}c_{11}) &
U b_{12}-c_{22}+a_{21}c_{12}+b_{12}a_{12}+\frac{1}{\lambda\mu}(c_{11}+c_{12}a_{12}) \\
c_{22}+\frac{1}{\lambda}b_{21}-\frac{1}{\lambda\mu}c_{11} &
\mu V b_{21}-c_{12}+c_{21}+c_{22}a_{12}
\end{pmatrix}.
\end{align*}

We want to find a tame automorphism of this algebra that sends $\pa^-$ to a stabilization of the differential for the trivial braid. We do this by
replacing generators one by one so that the new differential on all generators but $c_{12}$ and $e_{11}$ is trivial or nearly trivial, and the new differential on $c_{12}$ and $e_{11}$ is as in the case of the trivial 1-braid; for instance, since $\textstyle{\pa^-(c_{12}+\frac{1}{\mu}c_{11})} = U+\lambda+\lambda\mu V+\mu$, we can replace $c_{12}$ by $\textstyle{c_{12}-\frac{1}{\mu}c_{11}}$ to get a generator with differential $U+\lambda+\lambda\mu V+\mu$.
In full, if we apply successively the eight tame automorphisms
\begin{align*}
e_{12} &\mapsto \textstyle{
e_{12}-b_{12}c_{12}-\frac{1}{\lambda\mu} c_{12}^2} \\
c_{21} &\mapsto c_{21}-\mu Vb_{21}+c_{12}-c_{22}a_{12} \\
b_{12} &\mapsto \textstyle{
b_{12}-\frac{1}{\mu}c_{22}+\frac{1}{\lambda\mu^2} c_{11}} \\
b_{21} &\mapsto \textstyle{
b_{21}+\frac{1}{\mu}c_{11}-\lambda c_{22}} \\
e_{11} &\mapsto \textstyle{e_{11}+\frac{1}{\lambda\mu^2}c_{22}c_{11}-\frac{1}{\mu} e_{12}-\frac{1}{\lambda\mu} e_{22}+Ve_{21}} \\
c_{12} & \mapsto \textstyle{c_{12}-\frac{1}{\mu} c_{11}} \\
a_{21} &\mapsto \textstyle{a_{21}-V-\frac{1}{\mu}} \\
a_{12} &\mapsto a_{12}+\lambda+\lambda\mu V,
\end{align*}
then we obtain a DGA with the same generators but differential
\begin{gather*}
\pa^-(c_{11})=-\mu a_{12},
\qquad
\pa^-(c_{12}) = U+\lambda+\lambda\mu V+\mu,
\qquad
\pa^-(c_{22})= \mu a_{21}, \\
\pa^-(e_{12}) = -\mu b_{12},
\qquad
\pa^-(e_{21}) = \textstyle{\frac{1}{\lambda} b_{21}},
\qquad
\pa^-(e_{22}) = c_{21},\\
\pa^-(a_{12}) = \pa^-(a_{21}) = \pa^-(b_{12}) = \pa^-(b_{21}) = \pa^-(c_{21}) = \pa^-(e_{11}) = 0.
\end{gather*}
Destabilizing yields a DGA generated by $c_{12},e_{11}$ with differential $\pa^-(c_{12}) = U+\lambda+\lambda\mu V+\mu$, $\pa^-(e_{11})=0$, which agrees with the DGA for the trivial $1$-braid above.

For the $2$-braid $\sigma_1^{-1}$, we use the matrices
\[
\Phi^R_{\sigma_1^{-1}} = \begin{pmatrix} 0 & 1 \\
            	1 & -a_{21}
           \end{pmatrix}
,\quad
\Ll = \begin{pmatrix} \lambda \mu^{-1} & 0 \\
	0 & 1
      \end{pmatrix}.
\]
The expression for ${\pa}^- \mathbf{C}$ from Theorem~\ref{thm:mainlinktv} now yields in particular
\[
\pa^-(c_{11}) = \lambda\mu^{-1}+\lambda V+U a_{12}.
\]
We claim that the filtered DGA for $\sigma_1^{-1}$ differs from the filtered DGA for the trivial $1$-braid by setting $(U,V)=(0,0)$ and comparing $(\widehat{\widehat{\KCA}}, \widehat{\widehat\pa}).$
For $\sigma_1^{-1}$ we find that $\widehat{\widehat{\pa}}(c_{11}) = \lambda\mu^{-1}$, a unit in $\Z[\lambda^{\pm 1},\mu^{\pm 1}]$, and so the homology of $(\widehat{\widehat{\KCA}}, \widehat{\widehat\pa})$ is trivial. For the trivial $1$-braid, we obtain a DGA generated by $c,e$ with $\widehat{\widehat{\pa}}(c) = \lambda+\mu$ and $\widehat{\widehat{\pa}}(e) = 0$, and it is clear that this DGA has nontrivial homology.

\subsection{The knot $m(7_6)$}\label{sec:HF}
Let $K_1,K_2$ be the transverse knots given by the closures of the $4$-braids
\[
\sigma_1\sigma_2^{-1}\sigma_1\sigma_2^{-1}\sigma_3^{-1}\sigma_2\sigma_3^3,
\qquad
\sigma_1\sigma_2^{-1}\sigma_1\sigma_2^{-1}\sigma_3^3\sigma_2\sigma_3^{-1},
\]
respectively. These are both transverse representatives of the mirror
of the knot $7_6$, with self-linking number $-1$. The following result demonstrates that the filtered DGA is an effective invariant of transverse knots.

\begin{thm}
The filtered DGAs for $K_1$ and $K_2$ are not filtered stable tame isomorphic, and thus $K_1$ and $K_2$ are not transversely isotopic.
\end{thm}

\begin{proof}
Consider the DGAs $(\widehat{\KCA}(K_i),\widehat{\pa})$ over
$\Z[\lambda^{\pm 1},\mu^{\pm 1}]$ obtained from the filtered DGAs by
setting $(U,V)=(0,1)$. One can count by computer the number of DGA
maps (augmentations) from $(\widehat{\KCA}(K_i),\widehat{\pa})$ to
$(\Z/3,0)$ that send $\lambda$ to $-1$ and $\mu$ to $+1$; there are $5$ for $K_1$ and $0$ for $K_2$. The result
follows. (See \cite{Ng10} for more details.)
\end{proof}

As noted in \cite{Ng10}, the hat version of knot Floer homology for $m(7_6)$ is $0$ in
the relevant bidegree $(0,0)$, and so the transverse invariants
in knot Floer homology
\cite{LiscaOzsvathStipsiczSzabo09,OzsvathSzaboThurston08} do not
distinguish $K_1$ and $K_2$. One similarly finds that other previously developed transverse invariants (in Khovanov or Khovanov--Rozansky homology, for instance) do not distinguish $K_1$ and $K_2$. We conclude that the transverse invariant from knot contact homology is independent of previously known transverse invariants.

\def\cprime{$'$} \def\cprime{$'$}

\end{document}